\documentclass[11pt]{amsart}
\usepackage{amsmath,amsthm,amsfonts,amscd,amssymb,eucal,latexsym,mathrsfs,stmaryrd,enumitem}
\usepackage[all]{xy}

\newtheorem{theorem}{Theorem}[section]
\newtheorem{corollary}[theorem]{Corollary}
\newtheorem{lemma}[theorem]{Lemma}
\newtheorem{proposition}[theorem]{Proposition}

\theoremstyle{definition}
\newtheorem{definition}[theorem]{Definition}
\newtheorem{remark}[theorem]{Remark}

\newcounter{theoremintro}
\newtheorem{theoremi}[theoremintro]{Theorem}

\newcommand{\interior}{{\rm int}}

\newcommand{\cB}{{\mathscr B}}

\newcommand{\cU}{{\mathscr U}}
\newcommand{\cO}{{\mathscr O}}

\newcommand{\cT}{{\mathscr T}}
\newcommand{\cQ}{{\mathscr Q}}

\newcommand{\cP}{{\mathscr P}}

\newcommand{\cZ}{{\mathscr Z}}

\newcommand{\Zb}{{\mathbb Z}}

\newcommand{\Nb}{{\mathbb N}}

\newcommand{\eps}{\varepsilon}
\newcommand{\unit}{1}

\DeclareMathOperator{\diam}{diam}

\allowdisplaybreaks

\begin{document}

\title{Almost finiteness and the small boundary property}

\author{David Kerr}
\address{David Kerr,
Department of Mathematics,
Mailstop 3368, Texas A\&M University, 
College Station, TX 77843-3368, USA}
\email{kerr@math.tamu.edu}

\author{G{\'a}bor Szab{\'o}}
\address{G{\'a}bor Szab{\'o},
KU Leuven,
Department of Mathematics,  
Celestijnenlaan 200B, 
box 2400,
B-3001 Leuven, 
Belgium}
\email{gabor.szabo@kuleuven.be}

\date{June 10, 2019}

\begin{abstract}
Working within the framework of free actions of countable amenable groups on 
compact metrizable spaces, we show that the small boundary property is
equivalent to a density version of almost finiteness, which we call almost 
finiteness in measure, and that
under this hypothesis the properties of almost finiteness, 
comparison, and $m$-comparison for some nonnegative integer $m$
are all equivalent.
The proof combines an Ornstein--Weiss tiling argument with the use
of zero-dimensional extensions which are measure-isomorphic over singleton fibres.
These kinds of extensions are also employed to show
that if every free action of a given group on a zero-dimensional space
is almost finite then so are all free actions of the group 
on spaces with finite covering dimension. 
Combined with recent results of Downarowicz--Zhang and Conley--Jackson--Marks--Seward--Tucker-Drob
on dynamical tilings and of 
Castillejos--Evington--Tikuisis--White--Winter on the Toms--Winter conjecture, this
implies that crossed product C$^*$-algebras arising from
free minimal actions of groups with local subexponential growth
on finite-dimensional spaces are classifiable in the sense of Elliott's program.
We show furthermore that, for free actions of countably infinite amenable groups,
the small boundary property implies that the crossed product has uniform property $\Gamma$,
which confirms the Toms--Winter conjecture for such crossed products in the minimal case.
\end{abstract}

\maketitle

\tableofcontents

\section{Introduction}

A fundamental feature of the theory of amenability for groups is the connection 
it establishes between approximate invariance and tiling properties. This connection
was made clear in the work of Ornstein and Weiss \cite{OrnWei87}, who showed that the
approximate invariance that characterizes amenability via the F{\o}lner property
can be leveraged through a recursive procedure to produce tilings 
with approximately invariant shapes and almost full coverage, not only within the group itself but also in its 
probability-measure-preserving actions.
In the latter dynamical setting, these tilings (or tower decompositions, to use the more
customary terminology for actions) can be translated in a direct way into matrix 
subalgebras of the von Neumann algebra crossed product and thereby be used to 
produce a simple proof of hyperfiniteness when the action is free.
In particular, this shows that all free ergodic measure-preserving 
actions of countably infinite amenable groups on atomless standard probability spaces
give rise to the unique hyperfinite II$_1$ factor.\footnote{The first proof of this 
fact was given by Connes as an application of his 
celebrated result \cite{Connes76} that injectivity implies hyperfiniteness, 
whose full force is still needed to prove hyperfiniteness of the group 
von Neumann algebra itself.}

The attempt to similarly determine the structure of the C$^*$-crossed products
arising from actions of countably infinite amenable groups on compact
metrizable spaces has, like the general Elliott classification program for simple separable nuclear
C$^*$-algebras, been forced to contend with obstructions of
a topological nature that are conditioned by
the phenomenon of higher-dimensionality. Although the precise technical connections are still not so
well understood, these obstructions at the C$^*$-algebra level
are closely related to Gromov's notion of mean dimension \cite{Gro99}, 
which is an entropy-like invariant in topological dynamics
that provides a measure of asymptotic dimension growth.
Giol and the first author showed in \cite{GioKer2010} that the cubical subshifts 
constructed by Lindenstrauss and Weiss in \cite{LinWei00} as examples of minimal actions
that have nonzero mean dimension also produce crossed products which, for similar structural reasons, 
fail to behave well from the purview of Elliott's program.

Despite the complications in the C$^*$-picture caused by such obstacles,
it has become apparent over the last several years that,
from the viewpoint of both structure and classification theory, hyperfiniteness does
have a veritable analogue in the C$^*$-realm, 
namely the conjunction of nuclearity and $\cZ$-stability. 
For a simple separable unital nuclear C$^*$-algebra $A$,
the property of $\cZ$-stability can be expressed as the existence of 
order-zero completely positive contractive maps $\varphi$ 
from matrix algebras into $A$ such that the image of $\varphi$ is approximately central 
and $1-\varphi (1)$ is small in the sense of Cuntz subequivalence 
\cite{HirOro13}.\footnote{Whether or not this subequivalence is 
itself implemented in an approximately central way
is roughly what separates $\cZ$-stability from its specialization to the nuclear setting.}
The most recent affirmation of this analogy 
is the work of Castillejos, Evington, Tikuisis, White, and Winter 
on the Toms--Winter conjecture \cite{CasEviTikWhiWin18}, 
which shows that $\cZ$-stability can be substituted for
finite nuclear dimension as the regularity hypothesis 
in the C$^*$-counterpart to the classification of 
hyperfinite von Neumann algebras, which itself was only recently clinched in 
\cite{EllGonLinNiu15,GonLinNiu15,TikWhiWin17} after
several decades of effort by many researchers
(see Corollary~D and the accompanying discussion in \cite{TikWhiWin17}).

This parallel between the measure-theoretic and topological worlds
passes moreover to the dynamical level:
in \cite{Ker17} it was shown that 
a topological version of the Ornstein--Weiss tower decomposition called {\it almost finiteness} 
plays the role of $\cZ$-stability in a dynamical version of the Toms--Winter conjecture
and, when the action is free and minimal, implies that the crossed product is $\cZ$-stable.\footnote{Nuclearity is automatic in this case since the acting group is amenable.}
The main difference between almost finiteness and the Ornstein--Weiss decomposition 
is that the smallness of the remainder in the former is expressed using
dynamical subequivalence instead of a probability measure.
While towers and tower decompositions have
long played an important role in unraveling the structure of C$^*$-crossed products,
notably in the case of $\Zb$-actions \cite{Put90,LinPhi04,LinPhi10,TomWin13,EllNiu17}
and in the various theories of dimension for group actions
\cite{GueWilYu17,Sza15,SzaWuZac17,Ker17},
the novelty in the concept of almost finiteness is its dovetailing of topology
with measure-theoretic phenomena within the general setting of amenability. 
Indeed almost finiteness not only relates to the Ornstein--Weiss tower theorem
via analogy but also naturally
calls for the direct application of Ornstein--Weiss tiling methods 
towards C$^*$-algebraic ends, as illustrated by the proof of $\cZ$-stability 
in the free minimal case \cite{Ker17}.
This commingling of the measure-theoretic with the topological 
is very much in line with the kind of C$^*$-application of von Neumann algebra technology
that was pioneered by Matui and Sato in 
their groundbreaking work on the Toms--Winter conjecture \cite{MatSat12,MatSat14}.

One of the main aims of the present paper is to push this technical connection 
between measure and topology further at the dynamical level. 
Our work begins with the observation (Theorem~\ref{T-zero dim})
that the Ornstein--Weiss
tiling argument, as presented in \cite{ConJacKerMarSewTuc18}, applies equally well
to free actions of countable amenable groups on zero-dimensional spaces
so as to produce a disjoint collection of towers with clopen levels and F{\o}lner shapes 
such that the part of the space
that remains uncovered is small in upper density (or, equivalently, uniformly 
small on all invariant Borel probability measures). This motivates the concept
of {\it almost finiteness in measure} for free actions on general compact metrizable spaces,
which asks for the same kind of tower decomposition with remainder of small
upper density but only requires the levels to be open,
as is natural for the purpose of
accommodating spaces of higher dimension (Definition~\ref{D-af in measure}).
Almost finiteness in measure differs from almost finiteness in replacing the sense 
in which the remainder
is required to be small, from a topological one based on dynamical subequivalence
to a measure-theoretic one based on density. In this respect it is similar to 
the topological Rokhlin property for $\Zb^d$-actions introduced by Gutman in \cite{Gut11}.

By leveraging the above topological version of the Ornstein--Weiss tower theorem
via the use of extensions which are measure-isomorphic over singleton fibres,
whose basic theory is developed in Section~\ref{S-extn}, 
we prove in Theorem~\ref{T-af in measure} that, for free actions
on compact metrizable spaces, almost finiteness in measure is equivalent to the small
boundary property, i.e., the existence of a basis of open sets whose boundaries are null
for every invariant Borel probability measure. The small boundary property has its origins
in work of Shub and Weiss that initiated an investigation into the problem of when a given dynamical 
system possesses factors with lower entropy \cite{ShuWei91}.
It is known to imply mean dimension zero, and
is conjecturally equivalent to it under freeness, as well as to 
the decomposability of the action into an inverse limit of finite entropy actions.
The equivalence of these three properties for extensions of minimal $\Zb$-actions was established in 
\cite{LinWei00,Lin99} and was recently shown in \cite{Gut17} to hold for $\Zb$-actions with the marker property
and more generally in \cite{GutLinTsu16} for $\Zb^d$-actions with the marker property.
As a consequence of Theorem~\ref{T-af in measure}, we obtain the following as our first main result:

\begin{theoremi}[Theorem~\ref{T-af comparison}]
Let $G$ be a countably infinite amenable group.
For a free action $G\curvearrowright X$, the following are equivalent:
\begin{enumerate}
\item the action is almost finite,

\item the action has the small boundary property and comparison,

\item the action has the small boundary property and $m$-comparison for some $m\geq 0$.
\end{enumerate}
\end{theoremi}

\noindent 
This yields the corollary that, for free minimal actions on compact metrizable
spaces of finite covering dimension, finite tower dimension implies almost finiteness
(Corollary~\ref{C-tow dim af}).

The type of extensions employed in the proof of Theorem~\ref{T-af comparison}
are similarly mobilized to show the following reduction theorem.

\begin{theoremi}[Corollary~\ref{cor:findim-almost-finite}]\label{T-B}
Let $G$ be a countably infinite amenable group. 
Suppose that every free action of $G$ on a zero-dimensional
compact metrizable space is almost finite.
Then it follows that every free action of $G$ on a finite-dimensional
compact metrizable space is almost finite.\footnote{In fact, this is more generally proved for actions satisfying the so-called topological small boundary property, which is known to be automatic for free actions on finite-dimensional spaces by Theorem~3.8 of \cite{Sza15}.}
\end{theoremi}

Recently Downarowicz and Zhang have verified the hypothesis of Theorem~\ref{T-B} 
for groups $G$ whose finitely generated subgroups all have subexponential growth \cite{DowZha17}. 
Conley, Jackson, Marks, Seward, and Tucker-Drob also independently established this fact, in unpublished work,
by observing that it follows from a clopen version of the tiling argument in \cite{ConJacKerMarSewTuc18}
(see the introduction to Section~\ref{S-classification}). 
Together with the $\cZ$-stability theorem from \cite{Ker17}, 
work of Castillejos--Evington--Tikuisis--White--Winter 
\cite{CasEviTikWhiWin18}, and the classification results from \cite{EllGonLinNiu15,GonLinNiu15,TikWhiWin17},
this allows us to deduce our main result about C$^*$-crossed products:

\begin{theoremi}[Theorems~\ref{T-subexp} and \ref{T-classification}]
Let $G$ be a countably infinite group for which all finitely generated
subgroups have subexponential growth.
Then every free action $G\curvearrowright X$ on a compact metrizable finite-dimensional space 
is almost finite. 
As a consequence, the crossed products arising from
minimal such actions are classified 
by the Elliott invariant (ordered $K$-theory paired with tracial states)
and are simple ASH algebras of topological dimension at most 2.
\end{theoremi}

\noindent 
This classification theorem was known in the case of finitely generated groups $G$
having polynomial growth (which is equivalent to virtual nilpotence by Gromov's theorem \cite{Gro81})
as a consequence of \cite{SzaWuZac17,Bar17},
in which it was verified directly that the associated crossed products have finite nuclear dimension. 
In the polynomial growth setting one can also give a simpler proof of almost finiteness,
and hence of $\cZ$-stability, as we explain in Section~\ref{S-classification}.
The more general hypothesis of subexponential growth encompasses groups of
intermediate growth such as the Grigorchuk group \cite{Gri08}.

A companion paper to \cite{CasEviTikWhiWin18} 
shows that the Toms--Winter conjecture holds
under the additional hypothesis that the C$^*$-algebra satisfies a C$^*$-version
of Murray and von Neumann's property $\Gamma$, called 
\emph{uniform property} $\Gamma$ \cite{Casetal18}. 
That is, for nonelementary simple separable unital nuclear C$^*$-algebras
satisfying uniform property $\Gamma$, the properties of finite nuclear dimension,
$\cZ$-stability, and strict comparison are equivalent. Since $\cZ$-stability
implies uniform property $\Gamma$, this yields the implication from $\cZ$-stability to
finite nuclear dimension (the reverse implication 
had been previously established by Winter in \cite{Win12}).
To round out the paper we prove the following about C$^*$-crossed products 
within the general context of amenable acting groups:

\begin{theoremi}[Theorem~\ref{T-Gamma} and Corollary~\ref{C-TW}] 
Let $G$ be a countably infinite amenable group.
Let $X$ be a compact metrizable space and $G\curvearrowright X$ a free action 
with the small boundary property.
Then the crossed product $C(X)\rtimes G$ has uniform property $\Gamma$. 
In particular, the Toms--Winter conjecture holds for such crossed products when the action is also minimal.
\end{theoremi}

Our expectation is that the conditions in the Toms--Winter conjecture 
(i.e., finite nuclear dimension, $\cZ$-stability, and strict comparison)
are in fact automatic for crossed products of free minimal actions 
of countably infinite amenable groups
under the assumption of the small boundary property.
This is indeed what happens for $G = \Zb$ by a theorem of Elliott and Niu \cite{EllNiu17}. 
Their argument relies crucially on the recursive topological structure one obtains by generating tower decompositions via recurrence,
which is only available for $\Zb$, and so it appears that new methods 
will be required to tackle the case of $\Zb^2$ and other amenable groups.
\medskip

\noindent{\it Acknowledgements.}
The first author was partially supported by NSF grant DMS-1500593. 
A portion of the work was carried out during his six-month stay in 2017-2018 at the ENS de Lyon, 
during which time he held ENS and CNRS visiting professorships and was supported by Labex MILYON/ANR-10-LABX-0070.
He thanks Damien Gaboriau and Mikael de la Salle at the ENS for their generous hospitality.
The second author was supported by EPSRC grant EP/N00874X/1, the Danish National Research Foundation through the Centre for Symmetry and Deformation (DNRF92), the European Union's Horizon 2020 research and innovation programme under the grant MSCA-IF-2016-746272-SCCD, and a start-up grant of KU Leuven.
Part of the work was carried out during the second author's visit to Texas A\&M in June 2017, and he thanks the first author for his hospitality.
Both authors would like to thank Jianchao Wu for stimulating discussions on the subject of this paper.

\section{General terminology and notation}\label{S-notation}

Throughout the paper $G$ denotes a countable group. We are mainly interested in
the case when $G$ is amenable and will make this a blanket assumption
in Section~\ref{S-af in measure} while explicitly inserting it as a hypothesis
elsewhere when appropriate.

Given a finite set $K\subseteq G$ and a $\delta > 0$, we say that a finite
set $F\subseteq G$ is {\it $(K,\delta )$-invariant} if 
$|KF\Delta F | < \delta |F|$. The existence of $(K,\delta )$-invariant
finite sets for every finite $K\subseteq G$ and $\delta > 0$ is the F{\o}lner characterization
of amenability for $G$.
The notions of upper and lower (Banach) density for a subset $A$ of an amenable $G$, 
written $\overline{D} (A)$ and $\underline{D} (A)$,
are reviewed at the beginning of the next section.

For a compact metrizable space $X$ we write $M(X)$ for the set of Borel probability
measures on $X$ equipped with the weak$^*$ topology, under which it is compact.

A group action $G\curvearrowright X$ will be expressed using the concatenation
$(s,x)\mapsto sx$ except in some instances involving extensions where it will be convenient
to give it a name, such as $\alpha$, with the transformation $x\mapsto sx$ of $X$ for a given $s\in X$ 
sometimes written $\alpha_s$.

Actions on compact metrizable spaces are always assumed to be continuous.
For an action $G\curvearrowright X$ on a compact metrizable space we write $M_G (X)$
for the weak$^*$ closed subset of $M(X)$ consisting of all $G$-invariant Borel probability 
measures on $X$.

Let $\alpha : G\curvearrowright Z$ and $\beta : G\curvearrowright X$ be actions of $G$
on compact metrizable spaces. A map $\pi : Z\to X$ is a {\it factor map} if it is continuous and
surjective and satisfies $\pi (sz) = s\pi (x)$ for all $s\in G$ and $z\in Z$ (equivariance). 
We sometimes write $\pi : (Z,\alpha)\to (X,\beta)$ to indicate the precise roles of the actions.
The map $\pi$ is also called an {\it extension}. One also refers to $\beta$ as a {\it factor} of $\alpha$
and to $\alpha$ as an {\it extension} of $\beta$.

For the remainder of this section we let $G\curvearrowright X$ be a free action on a compact metrizable space. 
A {\it tower} is a pair $(V,S)$ consisting of a subset $V$ of $X$ and a finite subset $S$ of $G$
such that the sets $sV$ for $s\in S$ are pairwise disjoint. The set $V$ is the {\it base} of the tower,
the set $S$ is the {\it shape} of the tower, and the sets $sV$ for $s\in S$ 
are the {\it levels} of the tower.
The tower $(V,S)$ is {\it open} if $V$ is open and {\it clopen} if $V$ is clopen.
A {\it castle} is a finite collection of towers $\{ (V_i , S_i) \}_{i\in I}$
such that the sets $S_i V_i$ for $i\in I$ are pairwise disjoint.
The {\it remainder} of the castle is the set $X\setminus \bigcup_{i\in I} S_i V_i$.
The castle is {\it open} if each of the towers is open, 
and {\it clopen} if each of the towers is clopen.

Let $m\geq 0$ and $A,B\subseteq X$. Write 
$A\prec_m B$ if for every closed set $C\subseteq A$ there is a finite collection $\cU$
of open subsets of $X$ which cover $C$, an $s_U \in G$ for each $U\in\cU$, and a partition 
$\cU = \cU_0 \sqcup\cdots\sqcup \cU_m$ such that for each $i=0,\dots ,m$ the sets 
$s_U U$ for $U\in\cU_i$ are pairwise disjoint and contained in $B$.
When $m=0$ we also write $A\prec B$, in which case we say that $A$ is
{\it subequivalent} to $B$. The action $G\curvearrowright X$ 
has {\it $m$-comparison} if $A\prec_m B$ for all nonempty open sets $A,B\subseteq X$ 
satisfying $\mu (A) < \mu (B)$ for every $\mu\in M_G (X)$.
In the case $m=0$ we also simply say that the action has {\it comparison}.

The action $G\curvearrowright X$ is said to be {\it almost finite} if for every $n\in\Nb$, 
finite set $K\subseteq G$, and $\delta > 0$ there are 
\begin{enumerate}
\item an open castle $\{ (V_i ,S_i ) \}_{i\in I}$ whose shapes are $(K,\delta )$-invariant
and whose levels have diameter less than $\delta$, 

\item sets $S_i' \subseteq S_i$ such that $|S_i' | < |S_i |/n$ and
\[
X\setminus \bigsqcup_{i\in I} S_i V_i \prec \bigsqcup_{i\in I} S_i' V_i .
\]
\end{enumerate}
When $X$ is zero-dimensional this agrees with Matui's original notion of almost finiteness 
for second countable {\'e}tale groupoids with compact zero-dimensional unit spaces \cite{Mat12}.
In this case one can equivalently require that the levels of the tower
be clopen and partition $X$, so that condition (ii) can be dispensed with.
See Section~10 of \cite{Ker17}.

\section{Almost finiteness in measure and a topological Ornstein--Weiss tower theorem}\label{S-af in measure}

The proof of the Ornstein--Weiss tower theorem as presented in \cite{ConJacKerMarSewTuc18} shows,
by substituting clopenness for measurability, that every 
free action of an amenable group on a zero-dimensional compact metrizable space is 
almost finite in measure (Definition~\ref{D-af in measure}). 
For the convenience of the reader we will briefly reconstruct the argument 
as it applies to establish the clopen version, which we record 
as Theorem~\ref{T-zero dim}.
Along the way we will also establish a couple of facts (Propositions~\ref{P-D sup} and \ref{P-portmanteau})
that will be useful in later sections.

Throughout the section we assume $G$ to be amenable.

Let $G\curvearrowright X$ be an arbitrary action on a set.

\begin{definition}\label{D-density}
For a set $A\subseteq X$ we define
\begin{align*}
\underline{D} (A) = \sup_F \inf_{x\in X} \frac{1}{|F|} \sum_{s\in F} 1_A(sx)
\hspace*{7mm} \text{and} \hspace*{7mm}
\overline{D} (A) = \inf_F \sup_{x\in X} \frac{1}{|F|} \sum_{s\in F} 1_A(sx)
\end{align*}
where $F$ ranges over the nonempty finite subsets of $G$ in both cases.
These two quantities are called the {\it lower} and {\it upper (Banach) densities} of $A$.
\end{definition}

Observe that $\underline{D} (A) = 1 - \overline{D} (X\setminus A)$
and $\underline{D} (sA) = \underline{D} (A)$ for all $s\in G$.

The following is a dynamical version of Lemma~2.9 in \cite{DowHucZha16} and can be established
by the same argument.

\begin{lemma}\label{L-density Folner}
Let $(F_n )_n$ be a F{\o}lner sequence for $G$. Then 
\begin{align*}
\underline{D} (A) = \lim_{n\to\infty} \,\inf_{x\in X} \frac{1}{|F_n|} \sum_{s\in F_n} 1_A(sx)
\hspace*{7mm} \text{and} \hspace*{7mm}
\overline{D} (A) = \lim_{n\to\infty} \,\sup_{x\in X} \frac{1}{|F_n |} \sum_{s\in F_n} 1_A(sx).
\end{align*}
\end{lemma}

\begin{proposition}\label{P-D sup} 
Let $G\curvearrowright X$ be a continuous action on a compact metrizable space
and $A$ a subset of $X$. If $A$ is closed, then 
\begin{align*}
\overline{D} (A) = \sup_{\mu\in M_G (X)} \mu (A) ,
\end{align*}
while if $A$ is open then 
\begin{align*}
\underline{D} (A) = \inf_{\mu\in M_G (X)} \mu (A) .
\end{align*}
\end{proposition}

\begin{proof}
By the observation following Definition~\ref{D-density}, we need only prove the first assertion.
Suppose then that $A$ is closed.
By Lemma~\ref{L-density Folner} there are a F{\o}lner sequence $(F_n )_n$ for $G$ and
a sequence $\{ x_n \}$ in $X$ such that 
\begin{align*}
\overline{D} (A) = \lim_{n\to\infty} \frac{1}{|F_n |} \sum_{t\in F_n} 1_A(tx_n).
\end{align*}
Take a weak$^*$ cluster point $\mu$ of the sequence $\mu_n := |F_n |^{-1} \sum_{t\in F_n} \delta_{tx_n}$
in $M(X)$, where $\delta_{tx_n}$ denotes the point mass at $tx_n$. Then for every $f\in C(X)$
and $s\in G$ we have, writing $sf$ for the function $x\mapsto f(s^{-1} x)$
and abbreviating integration with respect to $\mu$ and $\mu_n$ as $\mu (\cdot )$ and $\mu_n (\cdot )$,
\begin{align*}
|\mu (sf) - \mu (f)| 
&\leq | \mu (sf) - \mu_n (sf) | + | \mu_n (sf) - \mu_n (f) | + | \mu_n (f) - \mu (f)| \\
&\leq | \mu (sf) - \mu_n (sf) | +  \frac{|s^{-1} F_n \Delta F_n|}{|F_n|} \| f \| + | \mu_n (f) - \mu (f)| \\
&\to 0
\end{align*}
as $n\to\infty$, showing that $\mu$ is $G$-invariant. Moreover, the portmanteau theorem yields
\begin{align*}
\mu (A) 
\geq \limsup_{n\to\infty} \mu_n (A) 
= \lim_{n\to\infty} \frac{1}{|F_n|} \sum_{t\in F_n} 1_A(tx_n)
= \overline{D} (A) 
\end{align*}
and hence $\overline{D} (A) \leq \sup_{\mu\in M_G (X)} \mu (A)$.

For the reverse inequality, it suffices to observe that for every $\mu\in M_G (X)$
and nonempty finite set $F\subseteq G$ we have 
\begin{align*}
\mu (A) 
= \frac{1}{|F|} \sum_{t\in F} \int_X 1_{t^{-1}A} \, d\mu 
&= \int_X \frac{1}{|F|} \sum_{t\in F} 1_{t^{-1} A} \, d\mu \\
&\leq \bigg\| \frac{1}{|F|} \sum_{t\in F} 1_{t^{-1} A} \bigg\|
= \sup_{x\in X} \frac{1}{|F|} \sum_{t\in F} 1_A(tx).
\end{align*}
so that $\mu (A) \leq \overline{D} (A)$.
\end{proof}

\begin{proposition}\label{P-portmanteau}
Let $X$ be a compact metrizable space with compatible metric $d$ and
let $A$ be a closed subset of $X$. Then for every
$\eps > 0$ there exists a $\delta > 0$ such that the set
\[
A_+ = \{ x\in X : d(x,A) \leq \delta \}
\]
satisfies $\overline{D} (A_+ ) \leq \overline{D} (A) + \eps$.
\end{proposition}

\begin{proof}
Suppose to the contrary that there is an $\eps > 0$ for which no suitable $\delta$ exists. 
Then by Proposition~\ref{P-D sup} we can find, for every $n\in\Nb$,
a $\mu_n \in M_G(X)$ such that the set
\[
A_n = \{ x\in X : d(x,A) \leq 1/n \}
\]
satisfies $\mu_n (A_n ) > \overline{D} (A) + \eps/2$. 
Since $M_G(X)$ is compact in the weak$^*$ topology there is a 
subsequence $(\mu_{n_k} )_k$ of $(\mu_n )_n$ which weak$^*$ converges to some $\mu\in M_G(X)$.
For integers $k\geq j\geq 1$ we have
\[
\mu_{n_k} (A_{n_j} ) \geq \mu_{n_k} (A_{n_k} ) > \overline{D} (A) + \eps/2 ,
\]
and so for a fixed $j$ the portmanteau theorem yields, since $A_{n_j}$ is closed,
\[
\mu (A_{n_j} ) \geq \limsup_{k\to\infty} \mu_{n_k} (A_{n_j} ) \geq \overline{D} (A) + \eps/2 .
\]
Since $A$ is equal to the decreasing intersection of the sets $A_{n_j}$ for $j\in\Nb$, we deduce that
\[
\mu (A) = \lim_{j\to\infty} \mu (A_{n_j} ) \geq \overline{D} (A) + \eps/2 .
\]
On the other hand $\mu (A)\leq \overline{D} (A)$ by Proposition~\ref{P-D sup}, producing a contradiction.
\end{proof}

For the rest of this section we will assume that the action $G\curvearrowright X$ is free.
Under this assumption, we can write the sum appearing in Definition~\ref{D-density}
as a cardinality via the equation
\[
\sum_{t\in F} 1_A(tx) = |A\cap Fx|
\]
for all $x\in X$, $A\subseteq X$ and finite subsets $F\subseteq G$.

\begin{definition}\label{D-af in measure}
A free action $G\curvearrowright X$ on a compact metric space is said to be
{\it almost finite in measure} if for every finite set $K\subseteq G$ and $\delta ,\eps > 0$ 
there is an open castle $\{ (V_i , S_i ) \}_{i\in I}$ with levels of diameter
less than $\delta$ such that
\begin{enumerate}
\item each shape $S_i$ is $(K,\delta )$-invariant,

\item $\underline{D} (\bigsqcup_{i\in I} S_i V_i ) \geq 1-\eps$.
\end{enumerate} 
\end{definition}

\begin{definition}
Let $K$ be a finite subset of $G$ and $\delta > 0$.
A set $A\subseteq X$ is said to be {\it $(K,\delta )^*$-invariant} if there exists a finite
set $F\subseteq G$ such that 
\begin{align*}
|(KA\Delta A) \cap Fx| < \delta |A\cap Fx|
\end{align*}
for all $x\in X$.
\end{definition}

\begin{definition}
Let $\eps > 0$. A collection $\{ A_i \}_{i\in I}$ of finite subsets of $G$ is said to be {\it $\eps$-disjoint}
if for every $i\in I$ there is a subset $A_i' \subseteq A_i$ with $|A_i' | \geq (1-\eps )|A_i |$ such that the sets
$A_i'$ for $i\in I$ are pairwise disjoint.
\end{definition}

One can establish the following in the same way as Lemma~\ref{L-density Folner}.

\begin{lemma}
Let $K$ be a finite subset of $G$ and $\delta > 0$. 
Let $( F_n )_n$ be a F{\o}lner sequence for $G$. 
Then a set $A\subseteq X$ is $(K,\delta )^*$-invariant if and only if 
$A\cap F_n x\neq\emptyset$ for all $x\in X$ and sufficiently large $n$ and
\begin{align*}
\lim_{n\to\infty} \, \sup_{x\in X} \frac{|(KA\Delta A) \cap F_n x|}{|A\cap F_n x|} < \delta .
\end{align*}
\end{lemma}

The following is Lemma~3.1 from \cite{ConJacKerMarSewTuc18}.

\begin{lemma}\label{L-1}
Let $K$ and $F$ be finite subsets of $G$ and let $\eps , \delta > 0$. Let 
$C$ be a subset of $X$ and for each $c\in C$ let $F_c$ be a 
$(K,\delta (1-\eps ))$-invariant subset of $F$ so that the  
collection $\{ F_c c : c\in C \}$ is $\eps$-disjoint and 
$\underline{D} (\bigcup_{c\in C} F_c c ) > 0$. Then $\bigcup_{c\in C} F_c c$ is
$(K,\delta )^*$-invariant.
\end{lemma}

The following is Lemma~3.2 from \cite{ConJacKerMarSewTuc18}.

\begin{lemma}\label{L-2}
Let $T$ be a finite subset of $G$ and let $\eps , \delta > 0$ be such that $\eps (1+\delta ) < 1$.
Let $B$ be a $(T^{-1} , \delta )^*$-invariant subset of $X$, and let $A$ be a subset of $X$ with 
$B\subseteq A$ such that $|A\cap Tx | \geq \eps |T|$ for all $x\in X$. Then
\begin{align*}
\underline{D} (A) \geq (1-\eps (1+\delta )) \underline{D} (B) + \eps .
\end{align*}
\end{lemma}

\begin{lemma}\label{L-clopen castle}
Let $X$ be a zero-dimensional compact metrizable space and 
$G\curvearrowright X$ a free action.
Let $Y$ be a clopen subset of $X$
and $S$ a finite subset of $G$. Let $0 < \eps < \frac12$ and $\delta > 0$. Then there 
is a clopen castle $\{ (V_i , S_i ) \}_{i=1}^n$ with levels of diameter less than $\delta$ such that 
\begin{enumerate}
\item $S_i \subseteq S$ and $|S_i | \geq (1-\eps )|S|$ for every $i=1,\dots ,n$,

\item the set $A = \bigsqcup_{i=1}^n S_i V_i$ 
satisfies $Y\cap A = \emptyset$, $Y\cup A = Y\cup \bigcup_{i=1}^n SV_i$, 
and $|Y\cup A\cap Sx| \geq \eps |S|$ for all $x\in X$.
\end{enumerate}
\end{lemma}

\begin{proof}
By freeness and zero-dimensionality we can find a clopen partition
$\{ V_1 , \dots , V_m \}$ of $X$ such that for each $i=1,\dots ,m$ 
the pair $(V_i , S)$ is a tower whose levels all have diameter less than $\delta$.
Write $\cT$ for the set of all $T \subseteq S$ such that $|T| \geq (1-\eps )|S|$.
Set $A_0 = Y$.
We will recursively construct sets $A_1 , \dots , A_m$ and clopen castles $\{ (V_{i,T} , T) \}_{T\in\cT}$
for $i=1,\dots ,m$, with some of the sets $V_{i,T}$ possibly being empty,
such that $A_i$ is the disjoint union of $A_{i-1}$ and 
$\bigsqcup_{T\in\cT} TV_{i,T}$ for each $i=1,\dots ,m$.

Let $1 \leq i\leq m$ and suppose that we have constructed the sets $A_1 , \dots , A_{i-1}$ 
and (in the case $i>1$) clopen castles 
$\{ (V_{j,T} , T) \}_{T\in\cT}$ for $j=1,\dots , i-1$ satisfying the required properties.
For each $T\in\cT$ define the clopen set
\begin{align*}
V_{i,T} = V_i \cap \bigg( \bigcap_{s\in S\setminus T} s^{-1} A_{i-1} \bigg)
\cap \bigg( \bigcap_{s\in T} (X\setminus s^{-1} A_{i-1} ) \bigg) .
\end{align*}
The sets $V_{i,T}$ are pairwise disjoint, and so 
$\{ (V_{i,T} , T) \}_{T\in\cT}$ is a clopen castle
such that $A_{i-1}$ and $\bigsqcup_{T\in\cT} TV_{i,T}$ are disjoint.
Put $A_i = A_{i-1} \sqcup \bigsqcup_{T\in\cT} TV_{i,T}$
to complete the recursive construction.

Set $A = \bigsqcup_{i=1}^m \bigsqcup_{T\in\cT} TV_{i,T}$. Then $Y\cap A = \emptyset$
and $Y\cup A = Y\cup \bigcup_{i=1}^m \bigcup_{T\in\cT} SV_{i,T}$.
Let $x\in X$. Then there is an $1\leq i\leq m$ such that $x\in V_i$.
If $x\in V_{i,T}$ for some $T\in\cT$ then $Tx\subseteq A_i \subseteq Y\cup A$
and hence 
\begin{align*}
|Y\cup A\cap Sx | \geq |Tx | = |T| \geq (1-\eps )|S|\geq \eps |S| .
\end{align*}
On the other hand, if $x\notin V_{i,T}$ for all $T\in\cT$ then the set of all $s\in S$ such that 
$sx\in A_i$ has cardinality at least $\eps |S|$ and hence 
$|Y\cup A\cap Sx |\geq |A_i \cap Sx | \geq \eps |S|$.
Thus the clopen castle $\{ (V_{j,T} , T) \}_{1\leq j\leq m, T\in\cT}$ satisfies condition (ii)
in the lemma statement. Condition (i) is built into the construction.
\end{proof}

The following is a simple (and well-known) exercise.

\begin{lemma}\label{L-almost invt}
Let $K$ be a finite subset of $G$ and $\delta > 0$. Then there is an $\eps > 0$ such that
if $F$ is a $(K,\eps )$-invariant finite subset of $G$ then every set $F' \subseteq F$
with $|F'| \geq (1-\eps )|F|$ is $(K,\delta )$-invariant.
\end{lemma}

\begin{theorem}\label{T-zero dim}
A free action $G\curvearrowright X$ on a zero-dimensional compact metrizable space 
is almost finite in measure.
\end{theorem}

\begin{proof}
Let $K$ be a finite subset of $G$ and $\delta > 0$. Let $\eps > 0$ 
be such that $\eps (1+\delta ) < 1$ and also such that it satisfies the
conclusion of Lemma~\ref{L-almost invt} with respect to $K$ and $\delta$.
Take an $n\in\Nb$ such that $(1-\eps )^n < \eps$.

Fix a $\beta > 0$ such that $(1+\beta )^{-1} (1-(1-(1+\beta )\eps )^n ) > 1 - \eps$ and
take nonempty $(K,\eps )$-invariant finite sets $F_1 , \dots , F_n \subseteq G$ such that 
for all $1\leq j < i \leq n$ the set $F_i$ is $(F_j^{-1} ,\beta (1-\eps ))$-invariant.
By a recursive procedure we will construct, for each $i$ running from $n$ down to $1$,
a clopen castle $\{ (V_{i,k} , F_{i,k} ) \}_{k=1}^{k_i}$ 
with levels of diameter less than $\delta$ 
and $F_{i,k} \subseteq F_i$ and $|F_{i,k} | \geq (1-\eps )|F_i |$ for $k=1,\dots ,k_i$
so that the sets $A_i = \bigsqcup_{k=1}^{k_i} F_{i,k} V_{i,k}$ are pairwise disjoint and
satisfy
\begin{align*}
\bigcup_{j=i}^n A_j
= \bigcup_{j=i+1}^n A_j \cup \bigcup_{k=1}^{k_i} F_i V_{i,k}
\end{align*}
and
\begin{align*}
\underline{D} (A_i \cup A_{i+1} \cup \cdots\cup A_n ) 
\geq \frac{1}{1+\beta} ( 1 - (1-\eps (1+\beta ))^{n+1-i} ) .
\end{align*}
Since each $F_{i,k}$ will be $(K,\delta )$-invariant by our choice of $\eps$, this will establish the theorem.
 
For the first stage of the construction, apply Lemma~\ref{L-clopen castle} with $Y=\emptyset$
and $S = F_n$ to obtain a clopen castle $\{ (V_{n,k} , F_{n,k} ) \}_{k=1}^{k_n}$ 
with levels of diameter less than $\delta$ such that 
\begin{enumerate}
\item $F_{n,k} \subseteq F_n$ and $|F_{n,k} | \geq (1-\eps )|F_n |$ for every $k=1,\dots ,k_n$,

\item the set $A_n = \bigsqcup_{k=1}^{k_n} F_{n,k} V_{n,k}$ satisfies 
$A_n = \bigcup_{k=1}^{k_n} F_n V_{n,k}$ and
$|A_n \cap F_n x| \geq \eps |F_n |$ for all $x\in X$.
\end{enumerate}
Apply Lemma~\ref{L-2} with $B = \emptyset$ and $A = A_n$ to obtain 
$\underline{D} (A_n) \geq \eps$.

Now let $1\leq i < n$ and suppose that we have carried out the stages of the construction
from $n$ down to $i+1$. Apply Lemma~\ref{L-clopen castle} with 
$Y = \bigsqcup_{j=i+1}^n A_j$ and $S = F_i$ to obtain a clopen
castle $\{ (V_{i,k} , F_{i,k} ) \}_{k=1}^{k_i}$ 
with levels of diameter less than $\delta$ such that 
\begin{enumerate}
\item $F_{i,k} \subseteq F_i$ and $|F_{i,k} | \geq (1-\eps )|F_i |$ for every $k=1,\dots ,k_i$,

\item the set $A_i = \bigsqcup_{k=1}^{k_i} F_{i,k} V_{i,k}$ satisfies 
$A_i \cap (A_{i+1} \cup\cdots\cup A_n ) = \emptyset$, 
$\bigsqcup_{j=i}^n A_j = \bigsqcup_{j=i+1}^n A_j \cup \bigcup_{k=1}^{k_i} F_i V_{i,k}$,
and 
$|\bigsqcup_{j=i}^n A_j \cap F_i x| \geq \eps |F_i |$ for all $x\in X$.
\end{enumerate}
By Lemma~\ref{L-1} the set $\bigsqcup_{j=i+1}^n A_j$
is $(F_i^{-1} , \beta )^*$-invariant
and so we can apply Lemma~\ref{L-2} 
with $B = \bigsqcup_{j=i+1}^n A_j$ and $A = \bigsqcup_{j=i}^n A_j$ to obtain 
\begin{align*}
\underline{D} (A_i \cup\cdots\cup A_n ) 
&\geq (1-\eps (1+\beta )) \underline{D} (A_{i+1} \cup\cdots\cup A_n ) + \eps \\
&\geq \frac{1}{1+\beta} ( 1-\eps (1+\beta ) - (1-\eps (1+\beta ))^{n+1-i} + \eps (1+\beta )) \\
&= \frac{1}{1+\beta} ( 1 - (1-\eps (1+\beta ))^{n+1-i} ) ,
\end{align*}
completing the $i$th stage of the construction.
\end{proof}

\section{Extensions measure-isomorphic over singleton fibres}\label{S-extn}

The proofs of Theorems~\ref{thm:SBP} and \ref{thm:TSBP-extension}
will hinge on the use of extensions which are measure-isomorphic
over singleton fibres (Definition~\ref{D-meas isom}), about which
we develop some basic facts here.

Let $Z$ and $X$ be compact metrizable spaces and $\pi :Z\to X$ a continuous surjection.
Write $S_\pi$ for the set of all $x\in X$ such that $\pi^{-1} (x)$ is a singleton. 
Then $S_\pi$ is a $G_\delta$ set, for if we fix a compatible metric on $X$ then we
can write $S_\pi$ as the intersection of the open sets $\{ x\in X : \diam (\pi^{-1} (x)) < 1/n \}$
for $n\in\Nb$. 

\begin{proposition}\label{P-bijective}
For $\mu\in M(Z)$ the following are equivalent:
\begin{enumerate}
\item $\pi_* \mu (S_\pi ) = 1$,

\item there is a basis $\cB$ for the topology on $Z$ such that for every $U\in\cB$ one has
$\pi_* \mu (\pi (U) \cap \pi (Z\setminus U)) = 0$,

\item for every open set $U\subseteq Z$ one 
has $\pi_* \mu (\pi (U) \cap \pi (Z\setminus U)) = 0$,

\item for every pair of disjoint closed sets $C_1 , C_2 \subseteq Z$ one 
has $\pi_* \mu (\pi (C_1 ) \cap \pi (C_2 )) = 0$,

\item for every closed set $C\subseteq Z$ one has $\pi_* \mu (\pi (C)) = \mu (C)$,

\item for every open set $U\subseteq Z$ one has $\pi_* \mu (\pi (U)) = \mu (U)$.
\end{enumerate}
\end{proposition}

\begin{proof}
(i)$\Rightarrow$(vi). For every open set $U\subseteq Z$ we have 
$\pi^{-1} (S_\pi \cap \pi (U)) = \pi^{-1}(S_\pi)\cap U$
and hence, using (i),
\begin{align*}
\mu(U) \leq \pi_* \mu (\pi (U)) = \pi_* \mu (S_\pi \cap \pi (U)) = \mu(\pi^{-1}(S_\pi)\cap U) =\mu (U) .
\end{align*}

(vi)$\Rightarrow$(v). Let $C$ be a closed subset of $Z$. Take a decreasing sequence  
$U_1 \supseteq U_2 \supseteq\dots$ of open subsets of $Z$ such that $C = \bigcap_{n=1}^\infty U_n$.
Then $\pi (C)$ is equal to the intersection of the decreasing sequence 
\[
\pi (U_1 )\supseteq \pi (U_2 ) \supseteq\dots
\]
and so (vi) yields
\begin{align*}
\pi_* \mu (\pi (C)) 
= \lim_{n\to\infty} \pi_* \mu (\pi (U_n )) 
= \lim_{n\to\infty} \mu (U_n ) 
= \mu (C) .
\end{align*}

(v)$\Rightarrow$(iv). Let $C_1$ and $C_2$ be disjoint closed subsets of $Z$. Then
since $\pi (C_1 )\cup \pi (C_2 ) = \pi (C_1 \cup C_2 )$ we can apply (v) to obtain
\begin{align*}
\pi_* \mu (\pi (C_1 ) \cap \pi (C_2 )) 
&= \pi_* \mu (\pi (C_1 )) + \pi_* \mu (\pi (C_2 )) - \pi_* \mu (\pi (C_1 )\cup \pi (C_2 )) \\
&= \mu (C_1 ) + \mu (C_2 ) - \mu (C_1 \cup C_2 ) 
= 0 .
\end{align*}

(iv)$\Rightarrow$(iii). Let $U$ be an open subset of $Z$.
Take an increasing sequence $C_1 \subseteq C_2 \subseteq\dots$ of closed
subsets of $Z$ such that $U = \bigcup_{n=1}^\infty C_n$.
By (iv), for every $n$ we have
\begin{align*}
\pi_* \mu (\pi (C_n ) \cap \pi (X\setminus U)) 
= 0 .
\end{align*}
Thus 
\begin{align*}
\pi_* \mu (\pi (U) \cap \pi (X\setminus U)) 
&= \pi_* \mu \bigg(\bigg(\bigcup_{n=1}^\infty \pi (C_n)\bigg) \cap \pi (X\setminus U)\bigg) \\
&= \pi_* \mu \bigg(\bigcup_{n=1}^\infty (\pi (C_n) \cap \pi (X\setminus U))\bigg) \\
&= \lim_{n\to\infty} \pi_* \mu (\pi (C_n) \cap \pi (X\setminus U)) 
=0 ,
\end{align*}
yielding (iii).

(iii)$\Rightarrow$(ii). Trivial. 

(ii)$\Rightarrow$(i). Let $\cB$ be as in (ii). Fix a compatible metric $d$ on $X$.
Given an $n\in\Nb$, we define the set $W_n = \{ z\in Z : \diam (\pi^{-1} (\pi (z))) \geq 1/n \}$,
which is readily seen to be closed. By compactness we can find open sets 
$U_1 , \dots , U_{k_n} \in\cB$ of diameter at most $1/(2n)$ whose union contains $W_n$.
Then for each $j=1,\dots , k_n$ we have $\pi (U_j ) \cap \pi (W_n ) \subseteq \pi (Z\setminus U_j )$,
and since by (ii) we have $\pi_* \mu (\pi (U_j ) \cap \pi (Z\setminus U_j )) = 0$ we deduce that
$\pi_* \mu (\pi (U_j ) \cap \pi (W_n )) = 0$ and hence $\mu (U_j \cap W_n ) = 0$.
Therefore
\begin{align*}
\mu (W_n ) \leq \sum_{j=1}^{k_n} \mu (U_j \cap W_n ) = 0 .
\end{align*}
It follows that $\mu (\bigcup_{n=1}^\infty W_n ) = \lim_{n\to\infty} \mu (W_n ) = 0$.
Since $\bigcup_{n=1}^\infty W_n$ is equal to $Z\setminus \pi^{-1} (S_\pi )$, we conclude that
$\pi_* \mu (S_\pi ) = 1$.
\end{proof}

Now let $G\curvearrowright Z$ and $G\curvearrowright X$ be actions on 
compact metrizable spaces and let $\pi : Z\to X$ be a factor map. 
Note that the $G_\delta$ set $S_\pi$ defined at the beginning of the section is $G$-invariant in this case.

\begin{definition}\label{D-meas isom}
The extension $\pi$ is {\it measure-isomorphic} if for every $\mu\in M_G (Z)$ the map 
$\pi$ induces a measure conjugacy 
between the measure-preserving actions $G\curvearrowright (Z,\mu )$ and $G\curvearrowright (X,\pi_* \mu )$,
i.e., $\pi$ restricts to an equivariant measure isomorphism between some $G$-invariant conull subsets
of $Z$ and $X$.
The extension $\pi$ is {\it measure-isomorphic over singleton fibres}
if $\pi_* \mu (S_\pi ) = 1$ for every $\mu\in M_G (Z)$.
\end{definition}

Note that if $\pi$ is measure-isomorphic over singleton fibres then it is measure-isomorphic,
as $\pi$ restricts to a $G$-equivariant Borel isomorphism 
$\pi^{-1} (S_\pi )\to S_\pi$ (\cite{Kec95}, Corollary~15.2).
The converse is false, as examples in \cite{DowGla16} demonstrate.

The following is immediate from Proposition~\ref{P-bijective}.

\begin{proposition}\label{P-bijective dyn}
For the extension $\pi : Z\to X$ the following are equivalent:
\begin{enumerate}
\item $\pi$ is measure-isomorphic over singleton fibres,

\item there is a basis $\cB$ for the topology on $Z$ such that for every $U\in\cB$ one has
$\pi_* \mu (\pi (U) \cap \pi (Z\setminus U)) = 0$ for all $\mu\in M_G (Z)$,

\item for every open set $U\subseteq Z$ one 
has $\pi_* \mu (\pi (U) \cap \pi (Z\setminus U)) = 0$ for all $\mu\in M_G (Z)$,

\item for every pair of disjoint closed sets $C_1 , C_2 \subseteq Z$ one 
has $\pi_* \mu (\pi (C_1 ) \cap \pi (C_2 )) = 0$ for all $\mu\in M_G (Z)$,

\item for every closed set $C\subseteq Z$ one has $\pi_* \mu (\pi (C)) = \mu (C)$ for all $\mu\in M_G (Z)$,

\item for every open set $U\subseteq Z$ one has $\pi_* \mu (\pi (U)) = \mu (U)$ for all $\mu\in M_G (Z)$.
\end{enumerate}
\end{proposition}

Note that if $G$ is amenable then one has $\pi_* (M_G (Z)) = M_G (X)$ as a consequence of the Hahn--Banach theorem.
To be more precise, we can view a given measure $\nu$ in $M_G (X)$ as a
$G$-invariant state $\sigma$ on the C$^*$-algebra $C(X)$ and, regarding $C(X)$ as a C$^*$-subalgebra
of $C(Z)$ under the embedding given by composition with $\pi$, we can extend $\sigma$ by the Hahn--Banach theorem 
to a state $\tilde{\sigma}$ on $C(Z)$.
Using compactness of the state space, we may then take a weak$^*$ cluster point of a sequence of averaged states given by
$f\mapsto |F_n |^{-1} \sum_{s\in F_n} \tilde{\sigma} (s^{-1} f)$ over a F{\o}lner sequence $\{ F_n \}$ of $G$, which yields a $G$-invariant state on $C(Z)$ restricting to $\sigma$ on $C(X)$.
By the Riesz representation theorem, such a state corresponds to a unique measure in $M_G (Z)$ that maps to $\nu$ under (the push-forward of) $\pi$.

In particular, one can replace ``$\pi_* \mu$'' and ``for all $\mu\in M_G (Z)$'' with ``$\nu$'' and
``for all $\nu\in M_G (X)$'', respectively, in each of the conditions (ii), (iii), and (iv)
in the above proposition.

\section{Equivalence of the small boundary property and almost finiteness in measure}\label{S-SBP af in meas}

Our goal in this section is to establish Theorem~\ref{T-af in measure} equating
almost finiteness in measure with the small boundary property under the assumption of freeness.
Below $X$ is assumed to be a compact metrizable space.

Let us begin by recalling the small boundary property. This concept originated in the work of Shub and Weiss 
(Section~1 of \cite{ShuWei91}) and later appeared in the context of actions with mean dimension zero \cite{LinWei00}.

\begin{definition}
An action $G\curvearrowright X$ is said to have the {\it small boundary property} if there is a basis for the topology on $X$ consisting of open sets $U$ such that $\mu(\partial U)=0$ for every $\mu\in M_G(X)$.
\end{definition}

When $G$ is amenable, the condition on $U$ above can be rephrased as $\overline{D}(\partial U)=0$ in view of Proposition \ref{P-D sup}.

Recall that a set $C\subseteq X$ is said to be {\it regular closed} if $C=\overline{\interior (C)}$.

\begin{definition}\label{D-regular}
A finite collection $\cP$ of subsets of $X$ is called a {\it regular closed partition} 
if all of its members are regular closed, 
$X=\bigcup\cP$, and $C_1\cap C_2\subseteq\partial C_1\cap \partial C_2$ for all 
distinct elements $C_1, C_2\in\cP$.
In this case we define the boundary of $\cP$ by
\[
\partial\cP = \bigcup_{C\in\cP} \partial C \ \subseteq \ X.
\]
If $\alpha$ is some homeomorphism from $X$ to itself, then $\alpha(\cP)$ will denote the regular closed partition given by the family $\{ \alpha(C) \mid C\in\cP \}$.
In other words, we identify $\alpha$ with its induced bijection on the set of all regular closed partitions on $X$.
\end{definition}

A regular closed partition is not a genuine partition in general, but can be viewed as
a partition modulo a set which is both closed and nowhere dense and hence 
topologically small in a manner of speaking.

\begin{remark}
Let $\cP_1$ and $\cP_2$ be two regular closed partitions of $X$. As for open covers, we say that $\cP_1$ 
{\it refines} $\cP_2$ if there exists a map $r: \cP_1\to\cP_2$ such that $C\subseteq r(C)$ for all $C\in\cP_1$.
We refer to such a map as a {\it refinement map}.
Unlike for open covers, it follows easily from the definition of a regular closed partition 
that such a refinement map, if it exists, is unique.
\end{remark}

\begin{definition}
Let $\cP_1$ and $\cP_2$ be two regular closed partitions of $X$.
We define their common refinement by
\[
\cP_1 \vee \cP_2 = \{ \overline{\interior(C_1\cap C_2)} \mid C_1\in\cP_1,\ C_2\in\cP_2 \},
\]
which is easily seen to be another regular closed partition.
\end{definition}

The main thrust of the following theorem lies in the construction proving the implication
\ref{thm:SBP:1}$\Rightarrow$\ref{thm:SBP:2}, which draws inspiration from a similar
construction in \cite{Kul95} carried out for $G=\mathbb Z$ in a slightly different context.

\begin{theorem}\label{thm:SBP}
Suppose that $G$ is amenable.
Let $\alpha: G\curvearrowright X$ be an action on a compact metric space.
The following are equivalent:
\begin{enumerate}[label=(\roman*)]
\item $\alpha$ has the small boundary property, \label{thm:SBP:1}
\item there exists an extension $\pi: (Z,\gamma) \to (X,\alpha)$ which is measure-isomorphic over singleton fibres such that $Z$ is zero-dimensional, \label{thm:SBP:2}
\item for every $\varepsilon>0$ there exists a finite collection $\cO$ of pairwise disjoint open subsets of $X$ of diameter at most $\eps$ such that $\underline{D}(\bigcup\cO)= 1$, \label{thm:SBP:3}
\item for every $\varepsilon>0$ there exists a finite collection $\cO$ of pairwise disjoint open subsets of $X$ of diameter at most $\eps$ such that $\underline{D}(\bigcup\cO)\geq 1-\varepsilon$.  \label{thm:SBP:4}
\item for every $\varepsilon>0$ and every pair of open sets $U, V\subseteq X$ with $\overline{U}\subseteq V$ there exists an open set $U_0$ with $U\subseteq U_0\subseteq V$ such that $\overline{D}(\partial U_0)\leq\varepsilon$. \label{thm:SBP:5}
\end{enumerate}
\end{theorem}

\begin{proof}
\ref{thm:SBP:1}$\Rightarrow$\ref{thm:SBP:2}.
Let us fix a compatible metric $d$ on $X$ for the course of the proof.
Given an $n\in\mathbb{N}$ we may find an open cover $\{ U_1, \dots , U_m \}$ of $X$ such that for each $j$ the set $U_j$ has diameter at most $1/n$ and $\nu(\partial U_j) = 0$ for all 
$\nu\in M_G (X)$.
By recursively defining
\[
C_1 = \overline{U}_1,\quad C_{k+1} = \overline{ U_{k+1}\setminus\Big( \bigcup_{j=1}^k C_k\Big)}
\]
for $k<m$, we obtain a regular closed partition $\cP_n = \{C_1,\dots,C_m\}$ of $X$.
By our choice of the sets $U_j$, each $C_j$ has diameter at most $1/n$ and $\nu(\partial\cP_n)=0$ for all 
$\nu\in M_G (X)$

Now let $F_1 \subseteq F_2 \subseteq\dots$ be an increasing sequence of finite subsets of $G$ with $e\in F_n=F_n^{-1}$ and $G=\bigcup_{n=1}^\infty F_n$.
Recursively define a sequence of regular closed partitions $\cQ_n$ of $X$ by
\[
\cQ_1 = \cP_1,\quad \cQ_{k+1} = \cP_{k+1} \vee \bigvee_{s\in F_{k}} \alpha_s(\cQ_k). 
\]
Given $k\in\mathbb N$ and $s\in F_k$, the partition $\alpha_s(\cQ_{k+1})$ refines $\cQ_k$ by definition, and we let $r_k^s: \alpha_s(\cQ_{k+1})\to\cQ_k$ be the refinement map sending each member of $\alpha_s(\cQ_{k+1})$ to the member of $\cQ_k$ which contains it.
For brevity, write $r_k^{1}=r_k$.
By the uniqueness of refinement maps between regular closed partitions, we may observe that
\begin{equation} \label{eq:r-composition-rule}
r_{k}^{st}\circ\alpha_{st}\circ r_{k+1} = r_{k}^{s}\circ\alpha_s\circ r_{k+1}^{t}\circ\alpha_t
\end{equation}
whenever $s,t\in G$ and $k$ are such that all of these maps are defined.

We define a compact metrizable space as the inverse limit
\[
Z=\lim_{\longleftarrow} \{ \cQ_k, r_k \} = \Big\{ (C_k)_{k\geq 1} \in \prod_{k\geq 1} \cQ_k \mid C_k = r_k(C_{k+1}) \text{ for all } k\geq 1\Big\}
\]
and denote by $r_{\infty, k}: Z\to\cQ_k$ the natural projection map for each $k$.
Clearly $Z$ is zero-dimensional, and in fact the collection
\[
\cB = \{ r_{\infty,k}^{-1}(A) \mid k\geq 1,\ A\in\cQ_k \}
\]
is a canonical basis for the topology consisting of clopen sets.
Note that each sequence $(C_k)_k\in Z$ is uniquely determined by any one of its cofinite tails, which we will use without further mention.

Given an element $z=(C_k)_k\in Z$, we observe that the intersection
$\bigcap_{k=1}^\infty C_k$ is nonempty by compactness and can only contain one element because its diameter 
is at most 
\[
\inf_{k\in\mathbb N} \operatorname{diam}(C_k) \leq \inf_{k\in\mathbb N} 1/k = 0.
\]
Thus we may define a map $\pi: Z\to X$ by declaring $\pi\big( (C_k)_k \big)$ to be the unique element 
in the intersection $\bigcap_{k=1}^\infty C_k$.

It is easy to check that $\pi$ is continuous, for if $z_1, z_2\in Z$ are both contained in the clopen set $r_{\infty,k}^{-1}(A)$ for some $k\geq 1$ and $A\in\cQ_k$, then by definition of $\pi$ this means $\pi(z_1), \pi(z_2)\in A$, which by construction of $\cQ_k$ implies $d(\pi(z_1),\pi(z_2))\leq 1/k$.

The surjectivity of $\pi$ holds because $(\cQ_k)_k$ is a sequence of successively finer regular closed partitions: if $x\in X$ then there is some $C_1\in\cQ_1$ with $x\in C_1$.
Inductively, if we have a finite sequence 
\[
C_1\supseteq C_2\supseteq\dots\supseteq C_{k_0} \ni x \quad\text{such that}\quad C_j\in\cQ_j,\quad 1\leq j\leq k_0,
\]
then we can find some $C_{k_0+1}\in\cQ_{k_0+1}$ with $x\in C_{k_0+1}\subseteq C_{k_0}$.
This gives us some sequence $(C_k)_k=z\in Z$ such that $x=\pi(z)$.

Let us now define a $G$-action $\gamma$ on $Z$.
Let $s\in G$.
Then there is a $k_0\in\mathbb N$ such that $s\in F_{k_0}$.
We define $\gamma_s: Z\to Z$ via the assignment
\[
 (C_k)_{k>k_0} \mapsto (r_k^s(\alpha_s(C_{k+1})))_{k\geq k_0} .
\]
Then $\gamma_s$ is a well-defined map on $Z$.
It is also continuous: if two elements $z_1, z_2\in Z$ agree on their first $j+1$ components as sequences for $j>k_0$, then it follows from the definition of $\gamma_s$ that $\gamma_s(z_1)$ and $\gamma_s(z_2)$ agree on their first $j$ components.
As $Z$ is equipped with the product topology coming from $\prod \cQ_k$, this yields continuity.

It is clear that $\gamma_e$ is the identity map.
To show that $\gamma$ is an action of $G$ by homeomorphisms, we only need to show that the assignment $s\mapsto\gamma_s$ is multiplicative in $G$ with respect to composition.
So let $s,t\in G$ and suppose that $s,t, st\in F_{k_0}$ for some $k_0 \in\Nb$.
Then
\begin{align*}
\gamma_{st}\big( (C_k)_{k>k_0} \big) 
&= (r_k^{st}(\alpha_{st}(C_{k+1})))_{k\geq k_0}  \\
&= (r_k^{st}(\alpha_{st}(r_{k+1}(C_{k+2}))))_{k\geq k_0}  \\
&\stackrel{\eqref{eq:r-composition-rule}}{=} r_k^s(\alpha_s(r_{k+1}^t(\alpha_t(C_{k+2}))))_{k\geq k_0}  \\
&= \gamma_s\big( (r_{k+1}^t(\alpha_t(C_{k+2})))_{k\geq k_0} \big) \\
&= (\gamma_s\circ\gamma_t)\big( (C_k)_{k>k_0} \big).
\end{align*}
This verifies that $\gamma$ is an action.

The map $\pi: Z\to X$ is then $G$-equivariant: for $s\in G$, $k_0\in\mathbb N$ as above and arbitrary $z=(C_k)_k\in Z$, we have
\[
\alpha_s(\pi(z))\in \alpha_s\Big( \bigcap_{k=1}^\infty C_k \Big) = \bigcap_{k\geq k_0}^\infty \alpha_s(C_{k+1}) \subseteq \bigcap_{k\geq k_0} r_k^s(\alpha_s(C_{k+1})) \ni \pi(\gamma_s(z)).
\]
Since the intermediate sets above are one-point sets, this yields $\alpha_s(\pi(z))=\pi(\gamma_s(z))$ for all $z\in Z$ and $s\in G$.

To show that the extension $\pi: (Z,\gamma)\to (X,\alpha)$ is measure-isomorphic over singleton fibres it suffices to check that $\pi$ satisfies condition (ii)
in Proposition \ref{P-bijective dyn} with respect to the canonical basis $\cB$ defined above.
Let $W\in\cB$ be an element in this basis.
Then there are $k\geq 1$ and $A\in\cQ_k$ such that $W=r_{\infty ,k}^{-1}(A)$.
The complement is given by $Z\setminus W=\bigsqcup_{C\in\cQ_k ,\, C\neq A} r_{\infty ,k}^{-1}(C)$.
By the definition of the map $\pi$, we see that $\pi(W)=A$ and $\pi(Z\setminus W)=\bigcup\{ C \mid C\in\cQ_k, C\neq A \}\subseteq X$.
As $\cQ_{k}$ is a regular closed partition the intersection of these two sets lies inside $\partial\cQ_k$, which has vanishing upper density by construction.
Therefore $\overline{D}(\pi(W)\cap\pi(Z\setminus W))=0$ for all $W\in\cB$, which finishes the proof of \ref{thm:SBP:2}.

\ref{thm:SBP:2}$\Rightarrow$\ref{thm:SBP:3}. Let $\pi: (Z,\gamma)\to (X,\alpha)$ be an extension as in \ref{thm:SBP:2}.
Let $\varepsilon>0$ be given. 
Let $\delta>0$ be small enough so that sets of diameter at most $\delta$ are mapped under $\pi$ to sets of diameter at most $\varepsilon$.
Choose a clopen partition $\cP$ of $Z$ consisting of sets with diameter at most $\delta$.
Let us consider 
\[
\cO = \{ \pi(W)\setminus\pi(Z\setminus W) \mid W\in\cP \}.
\]
By design, the members of $\cO$ are pairwise disjoint and have diameter at most $\varepsilon$.
For every $W\in\cP$ we have
$\overline{X\setminus\pi(W)}\subseteq \pi(Z\setminus W)$ and hence $\partial \pi(W) \subseteq \pi(W)\cap\pi(Z\setminus W)$, showing that the members of $\cO$ are open.

For each $W\in\cP$, since $Z\setminus W$ is closed we have $\nu(\pi(W)\cap\pi(Z\setminus W))=0$ for every 
$\nu\in M_G (X)$ by Proposition \ref{P-bijective dyn}(iv).
As $X$ is covered by the images of the members of $\cP$ under $\pi$, we have
\[
X\setminus\bigcup\cO \subseteq \bigcup \{ \pi(W)\cap\pi(Z\setminus W) \mid W\in\cP \}
\]
This implies that $\underline{D}(\bigcup\cO) = 1-\overline{D}(X\setminus\bigcup\cO) = 1$ and thus yields \ref{thm:SBP:3}.

\ref{thm:SBP:3}$\Rightarrow$\ref{thm:SBP:4}. Trivial.

\ref{thm:SBP:4}$\Rightarrow$\ref{thm:SBP:5}.
Let $U, V\subseteq X$ be open sets with $\overline{U}\subseteq V$.
Let $\varepsilon>0$ be given, and assume without loss of generality that it is small enough so that the $\varepsilon$-neighbourhood of $U$ is contained in $V$.
By \ref{thm:SBP:4} there is a finite collection $\cO$ of pairwise disjoint open subsets of $X$ with diameter less than $\varepsilon$ such that $\underline{D}(\bigcup\cO)\geq 1-\varepsilon$.
Set
\[
U_0 = U\cup\bigcup\{ W \in \cO \mid W\cap\partial U \neq \emptyset \}.
\]
Then clearly $U\subseteq U_0\subseteq V$.
By construction we have $\partial U_0 \subseteq X\setminus\bigcup\cO$ and thus $\overline{D}(\partial U_0)\leq\varepsilon$, as desired.

\ref{thm:SBP:5}$\Rightarrow$\ref{thm:SBP:1}. 
Let $U, V\subseteq X$ be open sets such that $\overline{U}\subseteq V$.
We will find an open set $W\subseteq X$ with $U\subseteq W\subseteq V$ such that $\overline{D}(\partial W)=0$.
Choose a sequence $(\eps_n )_n$ of strictly positive real numbers such that $\eps_n \to 0$.

Apply \ref{thm:SBP:5} to obtain an open set $W_1$ with $U\subseteq W_1\subseteq\overline{W}_1\subseteq V$ and $\overline{D}(\partial W_1)\leq\varepsilon_1/2$.
By Proposition~\ref{P-portmanteau} there is a closed neighbourhood $Z_1$ of $\partial W_1$ such that $Z_1 \subseteq V$ and $\overline{D}(Z_1) \leq \varepsilon_1$.
Apply \ref{thm:SBP:5} again to obtain an open set $W_2$ with $W_1\subseteq W_2\subseteq \overline{W}_2 \subseteq W_1\cup Z_1$ and $\overline{D}(\partial W_2)\leq \varepsilon_2/2$.
By Proposition~\ref{P-portmanteau} there is a closed neighbourhood $Z_2$ of $\partial W_2$ such that $Z_2 \subseteq Z_1$ and $\overline{D}(Z_2) \leq \varepsilon_2$.
Carry on recursively like this to generate for each $n$
an open set $W_n$ and a closed neighbourhood $Z_n$ of $\partial W_n$
such that $U\subseteq W_n\subseteq V$, $W_n\subseteq W_{n+1}$,
$Z_{n+1}\subseteq Z_n$, $W_{n+1}\subseteq W_n\cup Z_n$, and $\overline{D}(Z_n)\leq\varepsilon_n$.
Set
\[
W=\bigcup_{n\in\mathbb N} W_n.
\]
Then $U\subseteq W\subseteq V$. We also have $\partial W \subseteq \bigcap_{n\in\mathbb N} Z_n$ and hence $\overline{D}(\partial W)=0$.
This proves that $\alpha$ has the small boundary property.
\end{proof}

\begin{theorem}\label{T-af in measure}
Suppose that $G$ is amenable.
Then a free action $\alpha : G\curvearrowright X$ on a compact metrizable space
has the small boundary property if and only if it is almost finite in measure.
\end{theorem}

\begin{proof}
The ``if'' part follows from \ref{thm:SBP:4}$\Rightarrow$\ref{thm:SBP:1} of Theorem \ref{thm:SBP}.
Let us show the ``only if'' part.
Applying \ref{thm:SBP:1}$\Rightarrow$\ref{thm:SBP:2} of Theorem \ref{thm:SBP}, there exists an extension $\pi: (Z,\gamma)\to (X,\alpha)$ such that $Z$ is zero-dimensional and which is measure-isomorphic over singleton fibres.
Since $\alpha$ is free, so is $\gamma$. Fix compatible metrics on $Z$ and $X$.

Let $K$ be a finite subset of $G$ and $\eps > 0$.
By Theorem~\ref{T-zero dim} the action $\gamma$ is almost finite in measure, and so there exists a clopen castle 
$\{ (V_i , S_i ) \}_{i\in I}$ with $(K,\varepsilon)$-invariant shapes such that 
$\underline{D}(\bigcup_{i\in I} S_i V_i )\geq 1-\varepsilon$.
Choose a $\delta>0$ small enough so that subsets of $Z$ of diameter at most $\delta$ are mapped under $\pi$ to subsets of $X$ of diameter at most $\varepsilon$.
By partitioning each $V_i$ into smaller clopen sets and relabeling, we may assume that the levels of our castle all have diameter at most $\delta$.
For each $i\in I$ set
\[
W_i = \pi(V_i)\setminus\pi(Z\setminus V_i) 
\]
and note that $sW_i = \pi(sV_i)\setminus\pi(Z\setminus sV_i)$ for each $s\in S_i$.
Then $\{ (W_i , S_i ) \}_{i\in I}$ is an open castle.
The shapes of this castle are $(K,\varepsilon)$-invariant, and its levels have diameter at most $\varepsilon$ by our choice of $\delta$.
Also, since $\overline{D}(\pi(sV_i)\cap\pi(Z\setminus sV_i))=\overline{D}(\pi(V_i)\cap\pi(Z\setminus V_i)) = 0$ for every $i\in I$ and $s\in S_i$ by Propositions~\ref{P-bijective dyn} and \ref{P-D sup}, we have
\[
\underline{D}\bigg(\bigcup_{i\in I} S_i W_i\bigg) = \underline{D}\bigg(\bigcup_{i\in I} S_i V_i\bigg) \geq 1-\varepsilon.
\]
This finishes the proof.
\end{proof}

\begin{corollary}
Suppose that $G$ is infinite and amenable. Let $G\curvearrowright X$ be a free action
on a compact metrizable space which is almost finite in measure. 
Then the action has mean dimension zero.
\end{corollary}

\begin{proof}
By Theorem~\ref{T-af in measure},
the action has the small boundary property. This in turn implies mean dimension zero
by Theorem~5.4 of \cite{LinWei00}.\footnote{The arguments there are carried out for $G=\Zb$ 
but the authors make a point that they also apply more generally to all amenable groups.}
\end{proof}

\section{Comparison and almost finiteness}\label{S-comparison af}

In Theorem~\ref{T-af comparison} below
we strengthen Theorem~9.2 of \cite{Ker17}
using almost finiteness in measure via Theorem~\ref{T-af in measure}.
The novelty is the implication (iii)$\Rightarrow$(i)
under the hypothesis of the small boundary property, which generalizes the
corresponding implication in Theorem~9.2 of \cite{Ker17} and is established
below in the same way, only this time using Theorem~\ref{T-af in measure}
instead of the original form of the Ornstein--Weiss tower theorem.

Recall the definitions of almost finiteness and $m$-comparison from Section~\ref{S-notation}.
By Theorem~9.2 of \cite{Ker17}, almost finiteness implies comparison, 
and from this it is readily seen that the action $G\curvearrowright X$ is almost finite if and only if 
it is almost finite in measure (Definition~\ref{D-af in measure}) and has comparison.

\begin{theorem}\label{T-af comparison}
Suppose that $G$ is amenable.
Let $G\curvearrowright X$ be a free action on a compact metrizable space. 
Then the following are equivalent:
\begin{enumerate}
\item the action is almost finite,

\item the action has the small boundary property and comparison,

\item the action has the small boundary property and $m$-comparison for some $m\geq 0$.
\end{enumerate}
\end{theorem}

\begin{proof}
The implications (i)$\Rightarrow$(ii)$\Rightarrow$(iii) are part of Theorem~9.2 of \cite{Ker17}, with the small boundary property coming from Theorem~\ref{T-af in measure}.
Let us prove then (iii)$\Rightarrow$(i). 

Note first that if $G$ is finite then the small boundary property implies that $X$ is totally disconnected, 
in which case the freeness of the action implies that one has an equivariant decomposition
$X\cong Y\times G$, where $Y\subseteq X$ is a clopen subset and $G$ acts trivially on the first component and by left translation on the second component.
In this case the action is obviously almost finite. We may thus assume that $G$ is infinite.

By hypothesis the action has $m$-comparison for some $m\geq 0$.
Let $K$ be a finite subset of $G$, $\delta > 0$, and $n\in\Nb$. 
Set $q = (m+1)n$.
Since $G$ is infinite there are a finite set $K' \subseteq G$
with $K\subseteq K'$ and a $\delta' > 0$ with $\delta' \leq\delta$ such that
every nonempty $(K' , \delta' )$-invariant finite set $F\subseteq G$
is large enough so that there exists a set $S\subseteq F$ satisfying
$|F|/(2q) < |S| < |F|/q$.
Since the action has the small boundary property, by Theorem~\ref{T-af in measure} 
it is almost finite in measure, and so
we can find an open castle $\{ (V_i , T_i ) \}_{i\in I}$ with $(K',\delta' )$-invariant shapes
such that the remainder $A = X\setminus\bigsqcup_{i\in I} T_i V_i$ satisfies
$\overline{D} (A) < 1/(2q+1)$. By our choice of $K'$ and $\delta'$ we can find
pairwise disjoint sets $S_{i,0} , \dots , S_{i,m} \subseteq T_i$ 
such that each has the same cardinality $\kappa$ satisfying 
$|T_i|/(2q) < \kappa < |T_i|/q$.
Set $T_i' = S_{i,0} \sqcup\dots\sqcup S_{i,m}$ and note that
\begin{align}\label{E-n}
|T_i' | = (m+1)\kappa < \frac{m+1}{q} |T_i| = \frac1n |T_i | .
\end{align}

Next set $B = \bigsqcup_{i\in I} S_{i,0} V_i$.
By Proposition~\ref{P-D sup},
for every $\mu\in M_G (X)$ we have $\mu (A) \leq \overline{D} (A) < 1/(2q+1)$ 
and hence
\begin{align*}
\mu (B) 
\geq \sum_{i\in I} \frac{|S_{i,0} |}{|T_i|} \mu (T_i V_i ) 
> \frac{1}{2q} \mu \bigg( \bigsqcup_{i\in I} T_i V_i \bigg) 
= \frac{1}{2q} (1-\mu (A)) 
> \mu (A) .
\end{align*}
Therefore $A\prec_m B$ by our hypothesis of $m$-comparison, and so there exist a finite
collection $\cU$ of open subsets of $X$ covering $A$, an $s_U \in G$ for every $U\in\cU$,
and a partition $\cU = \cU_0 \cup\cdots\cup \cU_m$ such that for every $i=0,\dots ,m$
the sets $s_U U$ for $U\in\cU_i$ are pairwise disjoint and contained in $B$.

For every $i\in I$ and $j=0,\dots ,m$ choose a bijection $\varphi_{i,j} : S_{i,0} \to S_{i,j}$.
For $U\in\cU$, $i\in I$, and $t\in S_{i,0}$, write $W_{U,i,t} = U\cap s_U^{-1} tV_i$. Note that
each fixed $U$ is partitioned by the open sets $W_{U,i,t}$ for $i\in I$ and $t\in S_{i,0}$.
Also, if for each $U$ we write $j_U$ for the $j$ such that $U\in\cU_j$, then the sets
$\varphi_{i,j_U} (t)t^{-1} s_U W_{U,i,t}$ for $U\in\cU$, $i\in I$, and $t\in S_{i,0}$
are pairwise disjoint and contained in $\bigsqcup_{i\in I} T_i' V_i$.
We have consequently verified that $A\prec \bigsqcup_{i\in I} T_i' V_i$. In conjunction
with (\ref{E-n}), this establishes almost finiteness.
\end{proof}

Recalling the notion of tower dimension for actions from Definition~4.3 of \cite{Ker17},
we obtain the following strengthening of Theorem~7.2 in \cite{Ker17}.
Note that the assumption of finite covering dimension cannot be removed, as Example~12.5 in \cite{Ker17} illustrates.

\begin{corollary}\label{C-tow dim af}
Suppose that $G$ is amenable and that $X$ is a compact metrizable space with finite covering dimension.
Let $G\curvearrowright X$ be a free minimal action with finite tower dimension.
Then the action is almost finite.
\end{corollary}

\begin{proof}
Finite tower dimension and the finite-dimensionality of $X$ together imply,
by Theorem~7.2 of \cite{Ker17}, that the the action has $m$-comparison for some integer $m\geq 0$.
Since the action has the small boundary property in view of the finite-dimensionality of $X$
(as follows for example from Theorem~3.8 of \cite{Sza15}),
we can then apply Theorem~\ref{T-af comparison} to obtain the conclusion.
\end{proof}

\section{Almost finiteness and the topological small boundary property}

In Section~\ref{S-extn} we studied extensions which are measure-isomorphic 
over singleton fibres and used them in Section~\ref{S-SBP af in meas} to establish  
the equivalence between almost finiteness in measure
and the small boundary property. Here we apply such extensions
in the more purely topological setting of the so-called topological small boundary
property to obtain a result about almost finiteness (Theorem~\ref{thm:TSBP-almost-finite})
that sets the stage for Theorems~\ref{T-subexp} and \ref{T-classification} 
in the next section.

\begin{definition} \label{def:top-small}
Let $G\curvearrowright X$ be an action.
A set $A\subseteq X$ is said to be {\it topologically small} if there exists a constant $L\in\Nb$ such that whenever $s_0 , \dots , s_L$ are distinct elements of $G$ one has $s_0 A\cap\dots\cap s_L A = \emptyset$.
We say that the action has the {\it topological small boundary property} if there is a basis for the topology on $X$ consisting of open sets the boundaries of which are topologically small.
\end{definition}

\begin{remark}
It is clear from the definition that every topologically small set has vanishing upper density when $G$ is infinite.
In particular, the topological small boundary property implies the small boundary property for actions of infinite groups. Note also that the union of finitely many topologically small sets is again topologically small, a fact which will be used without further mention. 
\end{remark}

By closely examining the proof of \ref{thm:SBP:1}$\Rightarrow$\ref{thm:SBP:2} in Theorem \ref{thm:SBP} one sees that an analogous construction can be used to obtain the following.
Since the proof is very similar to before, we omit the details.

\begin{theorem} \label{thm:TSBP-extension}
Let $\alpha: G\curvearrowright X$ be an action with the topological small boundary property.
Then there exists an extension $\pi: (Z,\gamma)\to (X,\alpha)$ which is measure-isomorphic over singleton fibres such that $Z$ is totally disconnected and for all disjoint clopen sets $W_1, W_2\subseteq Z$ the intersection $\pi(W_1)\cap\pi(W_2)$ is topologically small.
\end{theorem}

\begin{lemma} \label{lem:almost-divisible}
Suppose that $G$ is infinite and amenable.
Let $G\curvearrowright X$ be a free action with the small boundary property.
Let $U\subseteq X$ be an open set with $\underline{D}(U)>0$.
Then for every $m\in\mathbb N$ and $\eta>0$ there exist pairwise disjoint open sets $U_1,\dots,U_m\subseteq U$ satisfying 
\[
\underline{D}(U_i) \geq \frac1m \underline{D}(U)-\eta
\] 
for all $i=1,\dots,m$. 
\end{lemma}

\begin{proof}
Let $m\geq 2$ and $\eta>0$ be given with $\eta<\underline{D}(U)$.
Setting $\theta = \underline{D}(U)-\eta/2$, by Proposition~\ref{P-portmanteau} (taking $A = X\setminus U$ there)
we may find an open set $W$ with $\overline{W} \subseteq U$ such that 
\begin{equation} \label{eq:W-theta}
\underline{D}(W)\geq \theta.
\end{equation}
Let $\delta>0$ be so small that the open $\delta$-neighbourhood of $\overline{W}$ is contained in $U$.
As $G$ is infinite, by our choice of $W$ we may find an $\varepsilon>0$ and a finite set $K\subseteq G$ such that every $(K,\varepsilon)$-invariant finite set $S\subseteq G$ has cardinality greater than $4m(1+\eta)/(\theta\eta )$ and satisfies 
\[
\inf_{x\in X} \frac{|Sx\cap W|}{|S|} \geq \frac{\theta}{1+\eta} ,
\]
in which case
\[
\inf_{x\in X} |Sx\cap W| \geq \frac{\theta}{1+\eta}|S| \geq \frac{4m}{\eta} .
\]

Since the action is almost finite in measure by Theorem~\ref{T-af in measure}, we can find an open castle consisting of towers $(V_j , S_j)$ for $j=1,\dots,N$ such that the shapes $S_j$ are $(K,\varepsilon)$-invariant, the levels all have diameter less than $\delta$, and such that
\begin{equation} \label{eq:castle-ld}
\underline{D}\big( \bigsqcup_{j=1}^N S_j V_j \big) \geq 1-\eta /4.
\end{equation}
It follows from the diameter condition that every level of the castle is either disjoint from $W$ or lies entirely in $U$.
Set
\begin{equation} \label{eq:Sj-prime}
S_j' = \{ s\in S_j : sV_j \cap W \neq \emptyset \} 
\end{equation}
and observe that
\[
|S_j'| \geq \inf_{x\in X} |S_jx\cap W| \geq \frac{4m}{\eta}.
\]
We may then find a partition
\[
S_j' = S_{j,1}'\sqcup\dots\sqcup S_{j,m}'
\]
with $|S_{j,i}'|\geq \lfloor |S_j'|/m \rfloor$, so that 
\[
\frac{|S_{j,i}'|}{|S_j'|} \geq \frac{1}{|S_j'|}\Big(\frac{|S_j'|}{m}-1\Big) \geq \frac1m - \frac{\eta}{4}
\]
for each $i=1,\dots,m$.
Set
\[
U_i = \bigsqcup_{j=1}^N \bigsqcup_{s\in S_{j,i}'} sV_j .
\]
By construction, we have for all $i=1,\dots,m$ and $\mu\in M_G (X)$  that
\begin{align*}
\mu(U_i) 
= \sum_{j=1}^N |S_{j,i}'| \mu(V_j) 
&= \sum_{j=1}^N \frac{|S_{j,i}'|}{|S_j'|} |S_j'| \mu(V_j) \\
&\geq \Big( \frac1m-\frac{\eta}{4} \Big) \sum_{j=1}^N |S_j'|\mu(V_j) \\
&\stackrel{\eqref{eq:Sj-prime},\eqref{eq:castle-ld}}{\geq} \Big( \frac1m-\frac{\eta}{4} \Big) \Big( \mu(W)-\frac{\eta}{4} \Big) .
\end{align*}
Appealing to Proposition~\ref{P-D sup} we thus conclude that
\begin{align*}
\underline{D}(U_i) 
\geq \Big( \frac1m-\frac{\eta}{4} \Big)\Big( \underline{D}(W)-\frac{\eta}{4} \Big) 
&\stackrel{\eqref{eq:W-theta}}{\geq} \Big( \frac1m-\frac{\eta}{4} \Big)\Big( \underline{D}(U)-\frac{3\eta}{4} \Big) \\
&\geq \frac1m \underline{D}(U) - \eta. \qedhere
\end{align*}
\end{proof}

\begin{lemma} \label{lem:top-small-implies-thin}
Suppose that $G$ is amenable.
Let $G\curvearrowright X$ be a free action with the small boundary property.
Let $U\subseteq X$ be an open set with $\underline{D}(U)>0$ and let $C\subseteq X$ be
a closed set which is topologically small. Then $C\prec U$.
\end{lemma}

\begin{proof}
Let $L\in\mathbb N$ be the smallness constant for $C$ as in Definition \ref{def:top-small}.
We will prove the assertion by induction in $L$.
First suppose that $L=1$, which means that the images of $C$ under the action of $G$ are pairwise disjoint.
As $\underline{D}(U)>0$, there exists some finite set $F\subseteq G$ with $\bigcup_{s\in F} sU=X$.
We may assume that $F=F^{-1}$ by replacing $F$ with $F\cup F^{-1}$.
Choose an open neighbourhood $V$ of $C$ such that the sets $sV$ for $s\in F$ are pairwise disjoint.
Then define $V_s=sU\cap V$ for $s\in F$, which yields an open cover of $C$ with the property that 
the sets $s^{-1}V_s$ for $s\in F$ are pairwise disjoint and contained in $U$. 
This verifies $C\prec U$ for the base step in the induction.

Now assume that the assertion holds for sets with smallness constant at most some given $L\in\Nb$
and let us show that it holds for sets with smallness constant $L+1$.
Let $C\subseteq X$ be a closed topologically small set with smallness constant $L+1$.

Applying Lemma \ref{lem:almost-divisible} repeatedly, we find pairwise disjoint open sets $U_s\subset U$ indexed by $s\in G$ with $\underline{D}(U_s)>0$ for all $s$.
For every $s\neq e$, the set $C\cap sC$ is clearly topologically small with constant $L$, and so by the induction hypothesis we have $C\cap sC\prec U_s$ for all $s\neq e$.
For every $s\neq e$ take an open neighbourhood $V_s$ of $C\cap sC$ such that $V_s \prec U_s$,
for example the union of a collection of open sets covering $C\cap sC$ which witness the subequivalence $C\cap sC\prec U_s$.

We note that the complement $C\setminus\bigcup_{s\in G\setminus\{e\}} sC$ is topologically small with constant $L=1$, and so the same is true of the closed set $C\setminus\bigcup_{s\in G\setminus\{e\}} V_s$, which it contains.
It follows by the base case of the induction hypothesis that $C\setminus \bigcup_{s\in G\setminus\{e\}} V_s \prec U_e$. Since
\begin{align*}
C &= \Big( \bigcup_{s\in G\setminus\{e\}} C\cap sC \Big)
\cup\Big( C\setminus \Big( \bigcup_{s\in G\setminus\{e\}} sC \Big) \Big) \\
&\subseteq \Big(\bigcup_{s\in G\setminus\{e\}} V_s \Big) \cup\Big( C\setminus\bigcup_{s\in G\setminus\{e\}} V_s \Big)
\end{align*}
we conclude using the compactness of $C$ that 
$C\prec \bigcup_{s\in G} U_s \subseteq U$.
\end{proof}

\begin{theorem}\label{thm:TSBP-almost-finite}
Suppose that $G$ is amenable,
and that every free action of $G$ on a zero-dimensional compact metrizable space is almost finite.
Then every free action of $G$ with the topological small boundary property is almost finite.
\end{theorem}

\begin{proof}
Let $X$ be a compact metrizable space
and $\alpha : G\curvearrowright X$ a free action with the topological small boundary property.
We apply Theorem \ref{thm:TSBP-extension} to obtain an extension $\pi: (Z,\gamma)\to(X,\alpha)$ which is measure-isomorphic over singleton fibres such that $Z$ is zero-dimensional and $\pi(W_1)\cap\pi(W_2)$ is topologically small for all pairs of disjoint clopen sets $W_1, W_2\subseteq Z$.

As $\alpha$ is free, it is clear that $\gamma$ is free.
By assumption, $\gamma$ is almost finite. Since $Z$ is zero-dimensional, it follows by the comments in the last paragraph 
of Section~\ref{S-notation} that given an $\varepsilon>0$ and a finite set $K\subseteq G$
we can find a clopen castle in $Z$ consisting of towers $(V_j , S_j)$ for $j=1,\dots,m$ so that each shape $S_j$ is $(K,\varepsilon)$-invariant and the castle partitions $Z$.
Without loss of generality, we may assume that $e\in S_j$ for all $j$.
Choose $\delta>0$ small enough so that sets of diameter at most $\delta$ in $Z$ are mapped under $\pi$ to sets of diameter at most $\varepsilon$ in $X$.
By subdividing the castle if necessary, we may assume that all of the tower levels have diameter at most $\delta$.

Define
\[
B_j = \pi(V_j)\setminus\pi(Z\setminus V_j).
\]
Then $(B_j , S_j )$ is an open tower in $X$ for the action $\alpha$, and together these towers form a castle
whose levels have diameter at most $\varepsilon$.
The remainder of the castle is contained in the closed set
\[
R = \bigcup_{j=1}^m \bigcup_{s\in S_j} \pi(sV_j )\cap\pi(Z\setminus sV_j ),
\]
which by our choice of $\pi$ is topologically small.
Since $\pi$ is measure-isomorphic over singleton fibres, by Proposition~\ref{P-D sup} and the fact that 
every $G$-invariant Borel probability measure on $Z$ assigns the levels of an open tower the same values we obtain
\[
\underline{D}\Big(\bigcup_{j=1}^m B_j\Big) = \underline{D}\Big(\bigcup_{j=1}^m V_j\Big) 
\geq \min_{1\leq j\leq m} \frac{1}{|S_j|} > 0 ,
\]
so that $R \prec \bigcup_{j=1}^m B_j$ by Lemma \ref{lem:top-small-implies-thin}. Thus $\alpha$ is almost finite.
\end{proof}

\begin{corollary}\label{cor:findim-almost-finite}
Suppose that $G$ is amenable,
and that every free action of $G$ on a zero-dimensional compact metrizable space is almost finite.
Then every free action of $G$ on a finite-dimensional compact metrizable space is almost finite.
\end{corollary}
\begin{proof}
This is a consequence of Theorem~\ref{thm:TSBP-almost-finite} and Theorem~3.8 of \cite{Sza15}.
\end{proof}

\section{Applications to classifiability}\label{S-classification}

Recently in \cite{DowZha17} Downarowicz and Zhang showed that if every finitely generated
subgroup of $G$ has subexponential growth then every free action of $G$ on
a zero-dimensional compact metrizable space is almost finite, i.e., 
there exist clopen castles which have F{\o}lner shapes and partition the space
(see the last paragraph of Section~\ref{S-notation}).
Conley, Jackson, Marks, Seward, and Tucker-Drob also observed independently that the tiling argument
of \cite{ConJacKerMarSewTuc18} can be adapted to the topological 
setting to establish this fact. In this case one first notes,
using the clopen version of the Ornstein--Weiss tower theorem (Theorem~\ref{T-zero dim}),
that the problem can be reduced to a matching argument on a bipartite graph. 
The graph is constructed from a clopen Ornstein--Weiss castle 
by using a fixed symmetric F{\o}lner set to connect points in the leftover part of the space, 
whose Banach density $d$ is taken to be extremely small, 
to points in some proportionally small collection of tower levels 
whose union has Banach density at least $2d$. 
With this graph one can progressively build better and better matchings through a process
of simultaneously flipping a given matching along a pairwise disjoint collection
of augmenting paths. Since subexponential growth implies that there is a 
uniform upper bound on the length of (minimal) augmenting paths, this process
will terminate after finitely many steps, resulting in a clopen matching. One thus obtains
a clopen castle which partitions the space, with the apppoximate invariance of the shapes
of the original castle being approximately preserved due to the proportional smallness 
of the collection of tower levels used to build the bipartite graph.

Combined with Corollary~\ref{cor:findim-almost-finite}, this result on almost finiteness
yields the following theorem. Note that the conclusion is still valid if we replace the hypothesis 
$\dim(X)<\infty$ with the weaker assumption that the action has the
topological small boundary, as we can appeal to Theorem~\ref{thm:TSBP-almost-finite}
instead.

\begin{theorem}\label{T-subexp}
Suppose that every finitely generated
subgroup of $G$ has subexponential growth. 
Let $X$ be a compact metrizable space with finite covering dimension.
Then every free action $G\curvearrowright X$ is almost finite.
\end{theorem}

For a free minimal action $G\curvearrowright X$ on a compact metrizable space, 
almost finiteness implies that $C(X)\rtimes G$ is $\cZ$-stable by Theorem~12.4 of \cite{Ker17}
and hence that $C(X)\rtimes G$ has finite nuclear dimension by Theorem~A of \cite{CasEviTikWhiWin18}.
In this case $C(X)\rtimes G$, given that it satisfies the UCT \cite{Tu99}, 
falls under the scope of the classification theorem recorded
as Corollary~D in \cite{TikWhiWin17}. This gives us the following result,
which again applies equally well if one assumes the topological small boundary property
on the action instead of $\dim(X)<\infty$.
The reader may compare this result to Theorem~8.8 of \cite{SzaWuZac17}.

\begin{theorem}\label{T-classification}
The crossed products of free minimal actions $G\curvearrowright X$
where $X$ is a compact metrizable space with finite covering dimension and
$G$ is such that each of its finitely generated subgroups has subexponential growth
are classified by the Elliott invariant (ordered $K$-theory paired with tracial states)
and are simple ASH algebras of topological dimension at most 2.
\end{theorem}
 
In the case that $G$ is finitely generated and has polynomial growth,
one can alternatively establish Theorem~\ref{T-subexp} by substituting
the following simpler proof of 
almost finiteness for free actions on zero-dimensional compact metrizable spaces.
We first observe, by considering the word-length metric with respect 
to a fixed finite generating set and using balls of sufficiently large radius, 
that such a $G$ satisfies the following property:
\begin{itemize}
\item[($\ast$)] there is a constant $c>0$
such that for every finite set $K\subseteq G$ and $\delta > 0$ there exists a nonempty 
finite set $F\subseteq G$ with $|F^{-1} F| \leq c|F|$ such that both
$F$ and $F^{-1} F$ are $(K,\delta )$-invariant.
\end{itemize}
It then suffices to combine the following two lemmas to obtain the desired conclusion.

\begin{lemma}\label{L-strong T}
Suppose that $G$ has property ($\ast$) with constant $c>0$. 
Let $G\curvearrowright X$ be a free action on a zero-dimensional compact metrizable space.
Then for all clopen sets $A,B\subseteq X$ satisfying $c\overline{D} (A) < \underline{D} (B)$
one has $A\prec B$.
\end{lemma}

\begin{proof}
Let $G\curvearrowright X$ be a free action on a zero-dimensional compact metrizable space and let
$A$ and $B$ be clopen subsets of $X$ satisfying $c\overline{D} (A) < \underline{D} (B)$.
Choose an $\eps > 0$ such that $c(\overline{D} (A) + \eps ) < \underline{D} (B) - \eps $.

Take a finite set $K\subseteq G$ and $\delta > 0$ such that for every $(K,\delta )$-invariant
finite set $E\subseteq G$ one has
$\overline{D} (A) + \eps \geq \sup_{x\in X} |A\cap Ex|/|E|$ and
$\underline{D} (B) - \eps \leq \inf_{x\in X} |B\cap Ex|/|E|$. By our choice of $c$,
there is a nonempty finite set $F\subseteq G$ such that $|F^{-1} F| \leq c|F|$ and  
both $F$ and $F^{-1} F$ are $(K,\delta )$-invariant.
Fix an enumeration $s_1 , \dots , s_n$ of the elements of $F$.
Set $A_1 = A\cap s_1^{-1} B$ and for $k=2,\dots , n$ recursively define 
\begin{gather*}
A_k = (A\setminus (A_1 \cup\cdots\cup A_{k-1} )) 
\cap s_k^{-1} (B\setminus (s_1 A_1 \cup\cdots\cup s_{k-1} A_{k-1} )) . 
\end{gather*}
Note that the sets $A_1 , \dots , A_k$ are clopen, pairwise disjoint, and contained in $A$. 

We claim that the union of the sets $A_1 , \dots , A_k$ is equal to $A$.
Suppose to the contrary that there exists an $x\in A\setminus\bigsqcup_{k=1}^n A_k$.
Set $I = \{ 1\leq i\leq n : s_i x \in B \}$. Then 
\begin{align}\label{E-lower I}
|I| = |Fx \cap B| \geq (\underline{D} (B) - \eps )|F|
\end{align}
Since $x$ does not belong to any of the sets $A_1 , \dots , A_n$, for every $i\in I$
there is a $k_i \in \{ 1,\dots ,i-1 \}$ and an $x_i \in A_{k_i}$ such that $s_i x = s_{k_i} x_i$,
which we can also write as $x_i = s_{k_i}^{-1} s_i x$.
Since the points $s_1 x, \dots , s_n x$ are distinct by freeness 
and the sets $A_1 , \dots , A_n$ are pairwise disjoint, the points
$x_i$ for $i\in I$ are distinct, which implies that the group elements $s_{k_i}^{-1} s_i$
for $i\in I$ are distinct. Thus
\begin{align*}
|I| 
\leq |A\cap F^{-1} Fx| 
&\leq (\overline{D} (A) + \eps ) |F^{-1} F| \\
&\leq (\overline{D} (A) + \eps ) c|F| \\ 
&< (\underline{D} (B) - \eps )|F| ,
\end{align*}
contradicting (\ref{E-lower I}). We must therefore have $A = \bigsqcup_{k=1}^n A_k$.
Since the sets $s_1 A_1 , \dots , s_n A_n$ are pairwise disjoint and contained in $B$,
we have thus shown that $A\prec B$.
\end{proof}

\begin{lemma}\label{L-zero af}
Suppose that $G$ is amenable.
Let $G\curvearrowright X$ be a free action 
on a zero-dimensional compact metrizable space.
Suppose that there exists a $c > 0$ such that for all clopen sets 
$A,B\subseteq X$ satisfying $c\overline{D} (A) \leq\underline{D} (B)$ one has $A\prec B$. 
Then the action is almost finite.
\end{lemma}

\begin{proof}
It is clear that the hypothesis in the theorem statement can be substituted for
$m$-comparison in the proof of (iii)$\Rightarrow$(i) in Theorem~\ref{T-af comparison}
given that the castle there can be taken to be clopen, as the proof of Theorem~\ref{T-zero dim}
demonstrates.
\end{proof}

\begin{remark}
Finally we point out that, for finitely generated groups, property ($\ast$)
is actually equivalent to polynomial growth. Indeed for finitely generated groups
polynomial growth is equivalent to virtual nilpotence by Gromov's theorem \cite{Gro81},
and property ($\ast$) (or even a weaker version of it, 
known as the Tempelman condition, which does not require $F^{-1} F$ to be $(K,\delta )$-invariant)
implies virtual nilpotence
by Corollary~11.2 in \cite{BreGreTao12}. This corollary shows that
if $G$ is a finitely generated group for which there is a $b>0$ such that 
for every finite set $E\subseteq G$ there exists a finite set $F\subseteq G$ 
satisfying $|F^{-1} F| \leq b|F|$ and $E\subseteq F$, then $G$ is virtually nilpotent 
(the corollary is stated in \cite{BreGreTao12} with $F^2$
in place of $F^{-1} F$, but the authors' application of their Corollary~1.7 in the proof
works equally well in the latter case). While the containment condition $E\subseteq F$
for a given $E$ does not appear in the definition of property ($\ast$),
we can arrange for the set $F$ in the definition
to satisfy this condition, 
for if $F$ is $(E,\delta )$-invariant for a sufficiently small $\delta > 0$ 
then it contains $Es$ for some $s\in G$, as is easy to check, 
in which case $Fs^{-1}$ contains $E$ and
satisfies $|(Fs^{-1})^{-1} Fs^{-1}| = |F^{-1} F|$ and $|Fs^{-1} | = |F|$.
\end{remark}

\section{Crossed products and the Toms--Winter conjecture}\label{S-TW}

The Toms--Winter conjecture asserts the equivalence of the following
three conditions for simple separable infinite-dimensional unital nuclear C$^*$-algebras:
\begin{enumerate}
\item finite nuclear dimension,

\item $\cZ$-stability,

\item strict comparison.
\end{enumerate}
Recently in \cite{Casetal18} the conjecture was settled in the case that the C$^*$-algebra 
has uniform property $\Gamma$, as defined below. Since $\cZ$-stability implies uniform property $\Gamma$
this yields the implication (ii)$\Rightarrow$(i) in full generality.
As finite nuclear dimension was known to imply $\cZ$-stability \cite{Win12}
one thereby obtains the equivalence of (i) and (ii).

\begin{definition}
A unital C$^*$-algebra $A$ with nonempty tracial state space
is said to have {\it uniform property $\Gamma$} if for every
finite set $\Omega\subseteq A$ and $\eps > 0$ there are two orthogonal
positive contractions $e_1, e_2 \in A$ such that
for all $a\in\Omega$ and $k=1,2$ one has
\begin{enumerate}
\item $\| e_k a - ae_k \| < \eps$, and

\item $|\tau (e_k a) - \frac{1}{2} \tau (a) | < \eps$ for every tracial state $\tau$ on $A$. 
\end{enumerate}
\end{definition}

In condition (i) above one can equivalently use the uniform trace norm $\| \cdot \|_{2,{\rm u}}$
(the supremum of the trace norms over all tracial states), as follows for example from Propositions~4.5 and 4.6 in \cite{KirRor14}.

Making use of tiling techniques as in the proofs of Theorem~5.3 in \cite{ConJacKerMarSewTuc18}
and Theorem~12.4 in \cite{Ker17} we will verify uniform property $\Gamma$ for crossed products
of free actions with the small boundary property, i.e., which are almost finite in measure.
The basic tool is the following disjointified version of the Ornstein--Weiss quasitiling theorem.
We say that a set $E\subseteq G$
is {\it tiled} by a collection $\cT$ of subsets of $G$
if $E$ can be partitioned into sets of the form $Tc$ where $T\in\cT$ and $c\in G$.

\begin{lemma}\label{L-tiling}
Suppose that $G$ is amenable. 
Let $K$ be a finite subset of $G$ and $\delta > 0$.
Let $0 < \eps < \frac12$. Then there are a finite set $K' \subseteq G$,
a $\delta' > 0$, and a finite collection $\cT$ of $(K,\delta )$-invariant finite subsets of $G$
such that for every $(K' , \delta' )$-invariant finite set $E\subseteq G$ 
there is a set $E' \subseteq E$ such that $|E'| \geq (1-\eps )|E|$ and $E'$
can be tiled by $\cT$.
\end{lemma}

\begin{proof}
As is simple to check, there exists an $\eta > 0$ such that for every
$(K,\eta )$-invariant finite set $T\subseteq G$ and every $T' \subseteq T$
with $|T' | \geq (1-\eta )|T|$ the set $T'$ is $(K,\delta )$-invariant.
By shrinking $\eps$ if necessary we may assume that $\eps\leq\eta$.
By the Ornstein--Weiss quasitiling theorem (see Theorem~4.36 of \cite{KerLi16}),
there are a finite set $K' \subseteq G$, a $\delta' > 0$,
and $(K,\eta )$-invariant finite sets $T_1 , \dots , T_m \subseteq G$
such that for every $(K' ,\delta' )$-invariant finite set $E\subseteq G$ there exist
$C_1 , \dots , C_m \subseteq G$ and pairwise disjoint sets $T_{i,c} \subseteq T_i$ 
with $|T_{i,c} | \geq (1-\eta )|T_i |$ for $i=1,\dots ,m$ and $c\in C_i$ such that 
\begin{enumerate}
\item $\bigcup_{i=1}^m T_i C_i$ is a subset of $E$ with cardinality at least $(1-\eps )|E|$, and 

\item $\bigsqcup_{i=1}^m \bigsqcup_{c\in C_i} T_{i,c} c = \bigcup_{i=1}^m T_i C_i$.
\end{enumerate}
By our choice of $\eta$, each of the sets $T_{i,c}$ is $(K,\delta )$-invariant.
We can thus take $\cT$ to be the union of all sets $T\subseteq G$ for which 
there is an $1\leq i\leq n$ such that $T\subseteq T_i$ and $|T| \geq (1-\eta )|T_i |$.
\end{proof}

\begin{lemma}\label{L-levels}
Let $X$ be a compact metrizable space
and $G\curvearrowright X$ a free action which is almost finite in measure.
Let $\cP$ be a finite regular closed partition of $X$ whose boundary $\partial\cP$ 
(Definition~\ref{D-regular}) has
zero upper density. Let $K$ be a finite subset of $G$ and $\delta > 0$. 
Then there is an open castle $\{ (V_i , S_i ) \}_{i\in I}$ such that
\begin{enumerate}
\item each shape $S_i$ is $(K,\delta )$-invariant,

\item each level of the castle is contained in the interior of 
some member of $\cP$ and has boundary of upper density zero,
 
\item $\underline{D} (\bigsqcup_{i\in I} S_i V_i ) \geq 1-\eps$.
\end{enumerate} 
\end{lemma}

\begin{proof}
By almost finiteness in measure there is an open castle $\{ (V_i , S_i ) \}_{i\in I}$ such that
each shape $S_i$ is $(K,\delta )$-invariant 
and $\underline{D} (\bigsqcup_{i\in I} S_i V_i ) \geq 1-\eps$.
The proof of Theorem~\ref{T-af in measure} shows that we may assume that the boundary
of each level of the castle has upper density zero, which is equivalent to saying that
$\overline{D} (\partial\bigsqcup_{i\in I} S_i V_i ) = 0$.
Write the interiors of the members of $\cP$ as $A_1 , \dots , A_n$.
Set $A_0 = X\setminus (A_1 \sqcup\cdots\sqcup A_n )$. Then the
sets $A_0 , \dots , A_n$ partition $X$.
Let $i\in I$. For each $x\in V_i$ consider the function 
$\sigma_x \in \{ 0,\dots , n \}^{S_i}$ such that $sx \in A_{\sigma_x (s)}$ for every $s\in S_i$.
Write $\Sigma_i = \{ 1,\dots , n \}^{S_i}$. For every $\sigma\in\Sigma_i$ 
define the set $V_{i,\sigma} = \{ x\in V_i : \sigma_x = \sigma \}$, which is open
because the sets $A_1 , \dots , A_n$ are open. The pairs $(V_{i,\sigma} , S_i )$
for $\sigma\in\Sigma_i$ thus form an open castle each of whose levels belongs
to the interior of some member of $\cP$, and the complement
$S_i V_i \setminus \bigsqcup_{\sigma\in\Sigma_i} S_i V_{i,\sigma}$
has upper density zero given that it is contained in $S_i S_i^{-1} \partial\cP$. 
Having done this for each $i\in I$, we obtain
an open castle $\{ (V_{i,\sigma} , S_i ) : i\in I, \, \sigma\in\Sigma_i \}$
such that the complement of the union of its levels is contained in the union of
$X\setminus \bigsqcup_{i\in I} S_i V_i$ and $\bigcup_{i\in I} S_i S_i^{-1} \partial\cP$,
and since the latter has upper density zero we deduce that
$\underline{D} (\bigsqcup_{i\in I} \bigsqcup_{\sigma\in\Sigma_i} S_i V_{i,\sigma} ) 
= \underline{D} (\bigsqcup_{i\in I} S_i V_i ) \geq 1-\eps$.
Finally, since the boundary of the set 
$\bigsqcup_{i\in I} \bigsqcup_{\sigma\in\Sigma_i} S_i V_{i,\sigma}$ is contained in
the union of $\partial\bigsqcup_{i\in I} S_i V_i$ and $\bigcup_{i\in I} S_i S_i^{-1} \partial\cP$,
it must have upper density zero, so that the boundary of each level of the castle we have 
constructed has upper density zero. Our requirements are thus fulfilled.
\end{proof}

In the proof of the following theorem we will abuse notation and use the same
symbol $\mu$ to denote a measure and the integral of a function with respect to that measure,
as well as the induced tracial state on the crossed product.

\begin{theorem}\label{T-Gamma}
Suppose that $G$ is infinite and amenable.
Let $X$ be a compact metrizable space and $G\curvearrowright X$ a free action 
with the small boundary property.
Then $C(X)\rtimes G$ has uniform property $\Gamma$.
\end{theorem}

\begin{proof}
To prove the theorem we will show that for all finite sets 
$\Omega \subseteq C(X)\rtimes G$ and $\eps > 0$
there exist positive contractions $f_1 , f_2\in C(X)$ with disjoint supports such that 
\begin{align}\label{E-half}
|\mu (f_k a) - \textstyle\frac12 \mu (a)| < \eps 
\end{align}
for all $k=1,2$, $a\in\Omega$, and $\mu\in M_G (X)$, and
\begin{align}\label{E-central}
\| f_k a - a f_k \| < \eps 
\end{align}
for all $k=1,2$ and $a\in\Omega$.

First of all, we claim that it is enough to show that for every finite regular closed partition 
$\cP$ of $X$ whose boundary has upper density zero (Definition~\ref{D-regular}), 
every finite set $L\subseteq G$,
and every $\eps > 0$ we can find positive 
contractions $f_1 , f_2\in C(X)$ with disjoint supports such that 
\begin{align}\label{E-half 2}
|\mu (f_k \unit_A) - \textstyle\frac12 \mu (A)| < \eps
\end{align}
for all $k=1,2$, $A\in\cP$, and $\mu\in M_G (X)$, and, denoting by $u_s$
the canonical unitary in $C(X)\rtimes G$ associated to a given group element $s$,
\begin{align}\label{E-central 2}
\| f_k u_s - u_s f_k \| < \eps
\end{align}
for all $k=1,2$ and $s\in L$.

Indeed first note that to verify (\ref{E-half})
it is enough that $\Omega$ be a subset of $C(X)$, for if $E$ denotes the 
canonical conditional expectation from $C(X)\rtimes G$ onto $C(X)$
then given any $f\in C(X)$ and $a\in C(X)\rtimes G$ we have
$\mu (fa) = \mu (E(fa)) = \mu (fE(a))$. 
Note next that by Theorem~\ref{thm:SBP} 
there is an extension $G\curvearrowright Z$ of $G\curvearrowright X$ which is measure-isomorphic over singleton fibres
such that $Z$ is zero-dimensional. Accordingly we can view $C(X)$ as a $G$-invariant 
unital C$^*$-subalgebra of $C(Z)$ and $C(X)\rtimes G$ as a unital C$^*$-subalgebra 
of $C(Z)\rtimes G$. By zero-dimensionality, 
for every finite collection of functions in $C(Z)$,
and in particular for every finite collection of functions in $C(X)$, there is a clopen
partition $\cP$ of $Z$ such that each function in the collection is approximately equal
in norm, to within a prescribed tolerance,
to a linear combination of the indicator functions of the members of $\cP$.
Given that the factor map $Z\to X$ induces a bijection $M_G (Z)\to M_G (X)$ via its push-forward,
a simple approximation argument then shows
that to verify (\ref{E-half}) we may in fact go outside
of $C(X)$ within the larger algebra $C(Z)$ and instead quantify $\Omega$ over the collections 
$\{ \unit_A : A\in\cP \}$ where $\cP$ ranges over the clopen partitions of $Z$.
But this then means that it is enough to verify (\ref{E-half 2}) with $\cP$
ranging, as in the claim, over the finite regular closed partitions
of $X$ whose boundaries have upper density zero, for 
under the factor map $Z\to X$ the image of every clopen partition of $Z$
is a regular closed partition of $X$ whose boundary has upper density zero,
as the proof of Proposition~\ref{thm:SBP}\ref{thm:SBP:2}$\Rightarrow$\ref{thm:SBP:3} shows. 
Finally, given that $C(X)$ is a commutative C$^*$-algebra one can replace
(\ref{E-central}) by (\ref{E-central 2}) via a simple approximation argument,
thereby establishing the claim.

So now let $\cP$ be a regular closed partition of $X$ with 
$\overline{D} (\partial\cP ) = 0$, $L$ a finite symmetric subset of $G$, and $0 < \eps < 1$.
Choose an integer $Q > 1/\eps$. 
By amenability there are a finite set $K\subseteq G$ and a $\delta > 0$ such that
every $(K,\delta )$-invariant finite set $E\subseteq G$
satisfies $|\bigcap_{s\in L^Q} sE | \geq (1-\eps /4) |E|$.
By Lemma~\ref{L-tiling} there are a finite set $K' \subseteq G$,
a $\delta' > 0$, and a finite collection $\cT = \{ T_1 , \dots , T_J \}$ 
of $(K,\delta )$-invariant finite subsets of $G$
such that for every $(K' , \delta' )$-invariant finite set $E\subseteq G$ 
there is a set $E' \subseteq E$ such that $|E'| \geq (1-\eps /6 )|E|$ and $E'$
can be tiled by $\cT$. Set $M = \max_{1\leq j\leq J} |T_j |$.

By Theorem~\ref{T-af in measure}, the action $G\curvearrowright X$ is almost finite in measure,
so by Lemma~\ref{L-levels} there is an open castle $\{ (V_i , S_i ) \}_{i\in I}$ such that
\begin{enumerate}
\item each shape $S_i$ is $(K',\delta' )$-invariant,

\item each level of the castle is contained in the interior of some member of $\cP$ 
and has boundary of upper density zero, and

\item $\underline{D} (\bigsqcup_{i\in I} S_i V_i ) \geq 1-\eps /6$. 
\end{enumerate} 
As $G$ is infinite,
we may demand even stronger approximate invariance in (i) so as to additionally force 
each of the shapes $S_i$ to have cardinality at least $6JM|\cP |^M /\eps$.

Let $i\in I$. By our invocation of Lemma~\ref{L-tiling},
there is a set $S_i' \subseteq S_i$ with $|S_i' | \geq (1-\eps /6)|S_i |$ and
sets $C_{i,1} , \dots , C_{i,J} \subseteq G$ such that the collection
$\{ T_j c : 1\leq j\leq J , \, c\in C_{i,j} \}$ partitions $S_i'$.
Let $1\leq j\leq J$. For each $c\in C_{i,j}$ write $\sigma_c$ for the element of
$\cP^{T_j}$ such that $tcV_i \subseteq \sigma_c (t)$ for all $t\in T_j$.
This element is unique by (ii) and the fact that the interiors of the members of $\cP$
are pairwise disjoint by the definition of regular closed partition.
For each $\sigma\in\cP^{T_j}$ write $C_{i,j,\sigma}$ for the set
of all $c\in C_{i,j}$ such that $\sigma_c = \sigma$, and choose two disjoint subsets
$C_{i,j,\sigma}^{(1)}$ and $C_{i,j,\sigma}^{(2)}$ of $C_{i,j,\sigma}$ with equal cardinality 
such that the complement of their union in $C_{i,j,\sigma}$, which we denote by $C_{i,j,\sigma}^{(0)}$, 
is either empty or a singleton, depending on the parity of the cardinality of $C_{i,j,\sigma}$.

Set $T_{j,Q} = \bigcap_{s\in L^Q} sT_j$, and note that, by our choice of $K$ and $\delta$,
\begin{align}\label{E-core}
|T_{j,Q} | \geq (1-\eps /4) |T_j| .
\end{align}
Next recursively define, for $q=0,\dots ,Q-1$,
\[
T_{j,q} = L^{Q-q} T_{j,Q} \setminus L^{Q-q-1} T_{j,Q} .
\]
The sets $T_{j,q}$ for $q=0,\dots ,Q$ are pairwise disjoint and contained in $T_j$,
and for $s\in L$ we have
\begin{align}\label{E-one}
sT_{j,Q} \subseteq T_{j,Q-1} \cup T_{j,Q}
\end{align}
and, for $q=1,\dots ,Q$,
\begin{align}\label{E-one 2}
sT_{j,q} \subseteq T_{j,q-1} \cup T_{j,q} \cup T_{j,q+1}
\end{align}

Let $i\in I$, $1\leq j\leq J$, and $c\in C_{i,j}$ be given.
Since the boundaries of our towers have upper density zero, by Proposition~\ref{P-portmanteau} 
there is an $\eta> 0$ such that the set $B_{i,j,c} := \{ x\in cV_i : d(x,X\setminus cV_i ) < \eta \}$
satisfies
\begin{align}\label{E-density SBP}
\mu (B_{i,j,c} ) < \frac{\eps}{4\sum_{i\in I} \sum_{j=1}^J |C_{i,j} |}  
\end{align}
for all $\mu\in M_G (X)$. Set $W_{i,j,c} = cV_i \setminus B_{i,j,c}$, which is closed. 
By Urysohn's lemma there is a continuous function $g_{i,j,c} : X \to [0,1]$ which is zero on the complement of $cV_i$
and one on $W_{i,j,c}$.
Given $1\leq j\leq J$ and $c\in C_{i,j}$ and
writing $\alpha_t$ for the automorphism of $C(X)$ that composes functions
with the transformation $x\mapsto t^{-1} x$, 
we then define the function
\begin{align*}
\tilde{g}_{i,j,c} &= \sum_{q=0}^Q \sum_{t\in T_{j,q}} \frac{q}{Q} \alpha_t (g_{i,j,c} ) .
\end{align*}
Now set
\begin{align*}
f_1 = \sum_{i\in I} \sum_{j=1}^J \sum_{c\in C_{i,j,\sigma}^{(1)}} \tilde{g}_{i,j,c} , 
\hspace*{8mm}
f_2 = \sum_{i\in I} \sum_{j=1}^J \sum_{c\in C_{i,j,\sigma}^{(2)}} \tilde{g}_{i,j,c} .
\end{align*}
For every $s\in L$ we have
\begin{align*}
u_s \tilde{g}_{i,j,c} u_s^{-1} - \tilde{g}_{i,j,c}
= \sum_{q=0}^Q \sum_{t\in T_{j,q}} \frac{q}{Q} \alpha_{st} (g_{i,j,c} )
-  \sum_{q=0}^Q \sum_{t\in T_{j,q}} \frac{q}{Q} \alpha_t (g_{i,j,c} )
\end{align*}
and so the function $u_s \tilde{g}_{i,j,c} u_s^{-1} - \tilde{g}_{i,j,c}$ has norm at most 
$1/Q$ by (\ref{E-one}) and (\ref{E-one 2}). Since these functions for different $i$, $j$ and $c$
have pairwise disjoint supports, we deduce that
\begin{align*}
\| u_s f_k u_s^{-1} - f_k \| \leq \frac{1}{Q} < \eps 
\end{align*}
for $k=1,2$, verifying the requirement (\ref{E-central 2}).

Finally, we check (\ref{E-half 2}). For $k=1,2$ we set 
\begin{align*}
Z_k = \bigsqcup_{i\in I} \bigsqcup_{j=1}^J \bigsqcup_{\sigma\in\cP^{T_j}} 
\bigsqcup_{c\in C_{i,j,\sigma}^{(k)}} T_j cV_i , \hspace*{7mm}
\tilde{Z}_k = \bigsqcup_{i\in I} \bigsqcup_{j=1}^J \bigsqcup_{\sigma\in\cP^{T_j}} 
\bigsqcup_{c\in C_{i,j,\sigma}^{(k)}} T_j cW_{i,j,c} .
\end{align*}
Let $A\in\cP$, $\mu\in M_G (X)$, and $k\in \{ 1,2 \}$.
Note first that $\mu (A \cap Z_1 ) = \mu (A\cap Z_2 )$ as we chose the sets
$C_{i,j,\sigma}^{(1)}$ and $C_{i,j,\sigma}^{(2)}$ to always have equal cardinality.
Thus
\begin{align}\label{E-ind}
\mu (f_k \unit_A )
\leq \mu (A\cap Z_k )
= \frac12 \mu (A\cap (Z_1 \sqcup Z_2 ))
\leq \frac12 \mu (A) .
\end{align}
Next set $R_0 = \bigsqcup_{i\in I} \bigsqcup_{j=1}^J \bigsqcup_{\sigma\in \cP^{T_j}}
\bigsqcup_{c\in C_{i,j,\sigma}^{(0)}} T_j c V_i$
and observe that for every $i\in I$ we have
\begin{align*}
\bigg| \bigsqcup_{j=1}^J \bigsqcup_{\sigma\in\cP^{T_j}} 
\bigsqcup_{c\in C_{i,j,\sigma}^{(0)}} T_j c \bigg|
\leq JM|\cP |^M
\leq \frac{\eps}{6} |S_i |
\end{align*}
and hence $\mu (R_0 ) \leq \eps /6$. Therefore
\begin{align*}
\mu (X\setminus (Z_1 \sqcup Z_2 )) 
&\leq \mu (R_0 ) + \mu \Big(\bigsqcup_{i\in I} (S_i \setminus S_i' )V_i \Big) 
+ \mu \Big( X\setminus \bigsqcup_{i\in I} S_i V_i \Big) \\
&\leq \frac{\eps}{6} + \frac{\eps}{6} + \frac{\eps}{6} \ \leq \ \frac{\eps}{2}
\end{align*}
and so
\begin{align*}
\mu (A\cap Z_k )
= \frac12 \mu (A\cap (Z_1 \sqcup Z_2 )) 
\geq \frac12 \mu (A) - \frac{\eps}{2} .
\end{align*}
Since the set
$R_k = \bigsqcup_{i\in I} \bigsqcup_{j=1}^J \bigsqcup_{\sigma\in \cP^{T_j}} 
\bigsqcup_{c\in C_{i,j,\sigma}^{(k)}} (T_j \setminus T_{j,Q} )c V_i$
satisfies $\mu (R_k ) \leq \eps /4$ by (\ref{E-core}), and we have
$\mu (Z_k \setminus \tilde{Z}_k ) < \eps /4$ by (\ref{E-density SBP}), it follows that
\begin{align*}
\mu (f_k \unit_A )
\geq \mu ((A\cap\tilde{Z}_k )\setminus R_k ) 
&\geq \mu (A\cap Z_k ) - \mu (Z_k \setminus \tilde{Z}_k ) - \mu (R_k ) \\
&\geq \frac12 \mu (A) - \eps .
\end{align*}
Combined with (\ref{E-ind}) 
this yields $|\mu (f_k \unit_A ) - \frac12 \mu (A) | < \eps$, verifying (\ref{E-half 2}).
This finishes the proof.
\end{proof}

In view of \cite{Casetal18} we deduce the following corollary.

\begin{corollary}\label{C-TW}
The Toms--Winter conjecture holds for the simple crossed product C$^*$-algebras 
in the statement of Theorem~\ref{T-Gamma}.
\end{corollary}

We expect that the conditions in the Toms--Winter conjecture are actually satisfied
by all of the simple crossed products in the statement of Theorem~\ref{T-Gamma}.
This is known to be the case when $G=\Zb$ by the work of Elliott and Niu \cite{EllNiu17}.

\begin{remark}
It is worth pointing out that the proof of Theorem~\ref{T-Gamma}
verifies a version of uniform property $\Gamma$ 
for the Cartan pair given by the inclusion $C(X) \subseteq C(X)\rtimes G$.
In other words, the elements witnessing uniform property $\Gamma$ for the crossed products in Theorem~\ref{T-Gamma} are actually coming from the diagonal subalgebra $C(X)$.

Stuart White has pointed out to us that the analogous property for Cartan pairs in von Neumann algebras has been studied before, and in fact played a crucial role in the first example of a II$_1$-factor with two nonconjugate Cartan subalgebras, due to Connes and Jones \cite{ConnesJones82}.
In particular, it is known in the von Neumann algebra context
that property $\Gamma$ for a Cartan pair can fail even if the ambient algebra itself has property $\Gamma$.

This gives rise to the speculation that uniform property $\Gamma$ for the C$^*$-algebraic Cartan pair $C(X)\subseteq C(X)\rtimes G$ may be closely tied to the small boundary property of the action, whereas we would expect that uniform property $\Gamma$ for the crossed product C$^*$-algebra alone may hold without any assumptions on the action used to construct it besides freeness.
\end{remark}

\end{document}